\documentclass[a4paper,11pt]{article}

\usepackage[utf8]{inputenc}
\usepackage[english]{babel}
\usepackage[T1]{fontenc}
\usepackage{microtype}
\usepackage{bm}

\usepackage[margin=2.cm]{geometry}
\usepackage{mathrsfs} 

\usepackage[all]{xy} 
\makeatletter
\let\1\@undefined 
\let\2\@undefined
\makeatother

\usepackage{float}
\usepackage{flafter} 
\usepackage[font=small]{caption}
\usepackage{subcaption}

\usepackage{tabu}
\usepackage{booktabs}

\usepackage{enumerate}
\usepackage[shortlabels]{enumitem}

\usepackage{xcolor}
\usepackage{graphicx}
\graphicspath{{../Figures/}}

\usepackage{mathsymb-2021-03-21}

\newenvironment{theorembis}[2][]
  {\renewcommand{\thetheorem}{\ref*{#2}$'$}%
   \addtocounter{theorem}{-1}%
   \begin{theorem}[#1]\label{#2 bis}}
  {\end{theorem}}

\usepackage{hyperref}
\hypersetup{
	colorlinks=true,
	citecolor=cyan,
	linkcolor=red, 
	urlcolor=cyan, 
	filecolor=yellow!50!black,
	breaklinks=true,
	pdfpagemode=UseNone,
	bookmarksopen=false,
}

\usepackage{amsthm} 

\usepackage{nameref}
\usepackage{cleveref} 
\newcommand*{\Crefp}[2]{\nameCref{#2}~\refp{#1}{#2}}

\usepackage{thm-2021-03-21}
\theoremstyle{definition}

\makeatletter
\DeclareRobustCommand{\labeltext}[2]{%
  \phantomsection
  #1\def\@currentlabel{\unexpanded{#1}}\label{#2}%
}
\makeatother

\newcommand{\note}[1]{}               

\newcommand*{\range}[2][1]{\{#1,\ldots,#2\}}

\newcommand{\thetas}[1][\theta_0]{\ensuremath{(#1,\bm \theta)}}
\newcommand{\thetahomo}[1][\theta_0]{$(#1,\bm \theta)$-homogeneous}
\newcommand{\diag}[1][\bm \theta]{$#1$-diagonal}

\newcommand{\an}[1][A(\bm z)]{[\bm{z^n}]#1}
\newcommand{\pid}{\m({\frac{1}{2\pi i}}^d}

\newcommand*{\dom}[1][\delta]{\Delta_{#1}}
\newcommand*{\ddom}[1][\delta]{\bm \Delta_{#1}}
\newcommand*{\odom}[1][\delta]{\bm \Omega_{#1}}
\newcommand*{\kdom}[1][\delta]{\bm K_{#1}}
\newcommand*{\poly}{\mathcal P}
\newcommand{\borel}{\mathcal B}
\newcommand{\laplace}[1][\bm c]{\mathcal L_{#1}}

\newcommand*{\cj}[1][j]{\mathcal C_{#1}}
\newcommand*{\cc}{\bm{\mathcal C}}
\newcommand*{\vj}[1][j]{V_{#1}^O}
\newcommand*{\vv}{\bm V}

\newcommand*{\nn}{n_0^{\bm \theta}}
\newcommand*{\nni}{n_0^{-\bm \theta}}
\newcommand*{\nj}[1][j]{\mathchoice{%
                   n_0^{\theta_{#1}}%
                }{ n_0\raisebox{1pt}{$\!\!^{\theta_{#1}}$}%
                }{ n_0^{\theta_{#1}}%
                }{ n_0^{\theta_{#1}}%
}}
\newcommand*{\nji}[1][j]{\mathchoice{%
                   n_0^{-\theta_{#1}}%
                }{ n_0\raisebox{1pt}{$\!\!^{-\theta_{#1}}$}%
                }{ n_0^{-\theta_{#1}}%
                }{ n_0^{-\theta_{#1}}%
}}

\renewcommand*{\Re}{\mathfrak{Re}}

\DeclareMathOperator*{\res}{Res}
\usepackage{esint} 

\usepackage[titles]{tocloft}
\title{A transfer theorem for multivariate $\Delta$-analytic functions\\with a power-law singularity}
\author{Linxiao~Chen}

\newcommand{\Addresses}{{
  \bigskip\footnotesize

\noindent  \textsc{ETH Z\"urich, Department of Mathematics, R\"amistrasse 101, 8092 Z\"urich, Switzerland}\par\nopagebreak
  \textit{E-mail address}: \texttt{linxiao.chen@math.ethz.ch}	\medskip
}}

\begin{document}

\maketitle

\begin{abstract}
This paper presents a multivariate generalization of Flajolet and Odlyzko's transfer theorem.
Similarly to the univariate version, the theorem assumes $\Delta$-analyticity (defined coordinate-wise) of a function $A(z_1,\ldots,z_d)$ at a unique dominant singularity $(\rho_1,\ldots,\rho_d) \in (\complex^*)^d$, and allows one to translate, on a term-by-term basis, an asymptotic expansion of $A(z_1,\ldots,z_d)$ around $(\rho_1,\ldots,\rho_d)$ into a corresponding asymptotic expansion of its Taylor coefficients $a_{n_1,\ldots,n_d}$. We treat the case where the asymptotic expansion of $A(z_1,\ldots,z_d)$ contains only power-law type terms, and where the indices $n_1,\ldots,n_d$ tend to infinity in some polynomially stretched diagonal limit.
The resulting asymptotic expansion of $a_{n_1,\ldots,n_d}$ is a sum of terms of the form
\begin{equation*}
I(\lambda_1,\ldots,\lambda_d) \cdot n_0^{-\Theta} \cdot \rho_1^{-n_1}\cdots \rho_d^{-n_d},
\end{equation*}
where $(\lambda_1,\ldots,\lambda_d) \in (0,\infty)^d$ is the direction vector of the stretched diagonal limit for $(n_1,\ldots,n_d)$, the parameter $n_0$ tends to $\infty$ at similar speed as $n_1,\ldots,n_d$, while $\Theta\in \real$ and $I:(0,\infty)^d \to \complex$ are determined by the asymptotic expansion of $A$.
\end{abstract}

\section{Introduction}\label{sec:introduction}

\paragraph{Univariate transfer theorems.}
In their pioneering work \cite{FlajoletOdlyzko1990}, Flajolet and Odlyzko developped a method, known as the transfer theorems, that allows one to compute a precise asymptotic expansion of a sequence $a_n$ as $n\to \infty$, from an asymptotic expansion of its generating function $A(z) = \sum_{n\ge 0} a_n z^n$ around some singularity $\rho \in \complex^* \equiv \complex \setminus \{0\}$. We use the notation $[z^n]A(z)$ for the coefficient $a_n = \frac{1}{n!}A\0n(0)$.
One distinctive feature of the transfer theorem is that it applies to generating functions that are \emph{$\Delta$-analytic}, that is, analytic on a \emph{$\Delta$-domain} of the form $\rho \cdot \dom \equiv \set{\rho z}{z\in \dom}$ for some $\delta>0$, where
\begin{equation}\label{eq:def Delta domain}
\dom := \Set{z\in \complex}{|z|<1+\delta, z\ne 1 \text{ and } \arg(1-z) \in (-\tfrac\pi2-\delta,\tfrac\pi2+\delta)} \,.
\end{equation}
See Figure~\refp{a}{fig:Delta/cone contour definition}.
(Simple extensions of the transfer theorem also apply to functions analytic in a finite intersection of $\Delta$-domains of the form $(\rho_1 \!\cdot\! \dom) \hspace{-1pt}\cap \hspace{-0.5pt} \cdots \hspace{-0.5pt} \cap\hspace{-1pt} (\rho_n \!\cdot\! \dom)$, but we shall not discuss that further here.)\\
Under the $\Delta$-analyticity assumption, $\rho$ is the unique dominant singularity of the function $A(z)$. By the change of variable $z'=z/\rho$, we can bring this singularity to $z=1$. In the following, we will assume \wlg\ that $\rho=1$.

In its most general form, Flajolet and Odlyzko's transfer theorem applies to functions whose asymptotic expansion is composed of any regular varying functions taken from a large class of ``standard scale functions'', such as $f(z)=(1-z)^\alpha (\log(1-z))^\beta$ for $\alpha,\beta\in \real$. Here we focus on functions whose asymptotic expansion contains only power-law terms of the form $(1-z)^\alpha$. In this simple case, the transfer theorem can be formulated as follows.

\begin{citetheorem}[Univariate transfer theorem]\label{thm:u-transfer}
Let $\alpha\in \real$ and let $A(z)$ be an analytic function on $\dom$ for some $\delta>0$.
\begin{enumerate}
\item (Analytic terms have exponentially small contribution)\\
If $A(z)$ is also analytic at $1$, then there exists $r\in (0,1)$ such that $[z^n]A(z) = O(r^n)$ when $n \to \infty$.
(In this case, $z=1$ is actually \emph{not} the dominant singularity of $A(z)$.)

\item (Coefficient asymptotics of power functions)\\
When $A(z)=(1-z)^\alpha$ with $\alpha \in \real$, we have the following asymptotic expansion as $n\to \infty$,
\begin{equation}\label{eq:u-transfer power expansion}
[z^n](1-z)^\alpha
=      \frac{n^{-\alpha-1}}{\Gamma(-\alpha)}
       \sum_{k=0}^\infty \frac{e_k(\alpha)}{n^k}
\equiv \frac{n^{-\alpha-1}}{\Gamma(-\alpha)}
       \m({ 1 + \frac{e_1(\alpha)}{n} + \frac{e_2(\alpha)}{n^2} + \cdots}
\end{equation}
where $e_k(\alpha) \in \rational[\alpha]$ is a polynomial of degree $2k$ in $\alpha$. More precisely, $e_k(\alpha) = \sum_{l=0}^k g_{k,l}\, \frac{\Gamma(-\alpha)}{\Gamma(-\alpha-k-l)}$, where $g_{k,l}$ is defined by the Taylor expansion
\begin{equation}\label{eq:u-transfer g-factors}
G(x,y):=\exp \m({ -y - \m({\frac yx+1}\log(1-x) } = \sum_{k=0}^\infty \m({ \sum_{l=0}^k g_{k,l}\,y^l } x^k \,.
\end{equation}

\item (Big-O transfer and little-o transfer)\\
If $A(z) = O((1-z)^\alpha)$ when $z\to 1$ in $\dom$, then
$[z^n]A(z) = O(n^{-\alpha-1})$ when $n\to \infty$.\\
If $A(z) = \,o\,((1-z)^\alpha)$ when $z\to 1$ in $\dom$, then
$[z^n]A(z) = \,o\,(n^{-\alpha-1})$ when $n\to \infty$.

\end{enumerate}
\end{citetheorem}

In practice, the transfer theorem is usually applied to functions $A(z)$ which are linear combinations of the three cases above, as in the following statement.

\begin{theorembis}[Univariate transfer theorem, integrated form]{thm:u-transfer}
Let $A(z)$ be an analytic function on $\dom$ for some $\delta>0$ such that when $z\to 1$ in $\dom$, we have
\begin{equation}\label{eq:u-transfer fun expansion}
A(z) = A\1{reg}(z) + h_0 (1-z)^{\alpha_0} + \cdots h_m (1-z)^{\alpha_m} + o\m({ (1-z)^{\alpha_m} } \,,
\end{equation}
where $A\1{reg}(z)$ is defined and analytic in a neighborhood of $1$ in $\complex$, and $h_j\in \complex$, $\alpha_j \in \real$ for all $0\le j\le m$.
Then, the coefficients of $A(z)$ has the following expansion when $n\to \infty$
\begin{equation}\label{eq:u-transfer coeff expansion}
[z^n] A(z) = h_0\cdot [z^n](1-z)^{\alpha_0} + \cdots + h_m\cdot [z^n](1-z)^{\alpha_m} + o(n^{-\alpha_m-1})
\end{equation}
where each term $[z^n](1-z)^{\alpha_j}$ has the asymptotic expansion given in Theorem~\refp2{thm:u-transfer}.\\
The same result holds if the little-o estimates in both \eqref{eq:u-transfer fun expansion} and \eqref{eq:u-transfer coeff expansion} are replaced by big-O estimates.
\end{theorembis}

\begin{remark*}~
\begin{enumerate}[nolistsep,topsep=0.5ex,itemsep=1ex]
\item
\Cref{thm:u-transfer} and \Cref{thm:u-transfer bis} are not exactly equivalent, for the following reasons:
When $A(z)$ is analytic at $1$, Theorem~\ref{thm:u-transfer} asserts that $[z^n]A(z)$ decays exponentially as $n\to \infty$, while Theorem~\ref{thm:u-transfer bis} only implies a super-polynomial decay.
On the other hand, one cannot prove Theorem~\ref{thm:u-transfer bis} by simply applying Theorem~\ref{thm:u-transfer} to each term of the expansion \eqref{eq:u-transfer fun expansion}, because the terms $A\1{reg}(z)$ and $o((1-\alpha)^{\alpha_m})$ in \eqref{eq:u-transfer fun expansion} have \emph{a priori} no analytic continuations on the whole domain $\dom$.

\item
For a sequence of functions $s_k(n)$, the asymptotic expansion $S(n) = \sum_{k=0}^\infty s_k(n)$ ususally means that for all $m\ge 0$, we have $S(n)=s_0(n)+\cdots + s_{m-1}(n) + O(s_m(n))$ as $n\to \infty$.
Theorem~\refp2{thm:u-transfer} uses a slightly modified notion of asymptotic expansion: we choose implicitly the family $s_\beta(n)=n^{-\beta}$ ($\beta \in \real$) as the reference asymptotic scale, and view each asymptotic expansion as a weighted sum of the form $S(n) = \sum_{k=0}^\infty c_k \cdot n^{-\beta_k}$ with some $\beta_0<\beta_1<\cdots$. The difference is that now the prefactors $c_k$ may vanish for some or even all $k\ge 0$, and the expansion should be read as: for all $m\ge 0$, we have $S(n) = c_0 n^{-\beta_0} + \cdots + c_{m-1} n^{-\beta_{m-1}} + O(n^{-\beta_m})$ as $n\to \infty$.

\item
It is also possible to include $\alpha \in \complex \setminus \real$ in \Cref{thm:u-transfer,thm:u-transfer bis}. The asymptotic expansion \eqref{eq:u-transfer power expansion} would hold unmodified. But complex exponents $\alpha$ rarely appear in applications and they complicate the asymptotic scale $s_\beta(n) = n^{-\beta}$ (since $\complex$ is not totally ordered). We restrict ourselves to $\alpha \in \real$ for the sake of simplicity.

\item
Theorem~\refp{1)--(2}{thm:u-transfer} together imply that the prefactor $\frac{1}{\Gamma(-\alpha)}$ in \eqref{eq:u-transfer power expansion} must vanish when $\alpha \in \natural$. Indeed, $\alpha \mapsto \frac{1}{\Gamma(-\alpha)}$ is an entire function and we have $\frac{1}{\Gamma(-\alpha)}=0$ \emph{\Iff} $\alpha \in \natural$.
\end{enumerate}
\end{remark*}

The transfer theorems apply to a large class of naturally occuring generating functions. For example, it is well known that all D-finite functions analytic at $0$ are linear combinations of $\Delta$-analytic functions. Among them, the algebraic functions always have singularities of power-law type, to which \Cref{thm:u-transfer bis} applies.
Compared to alternative methods such as Darboux's method or the Tauberian theorems, the transfer theorem has the advantage of giving a transparent correspondance between the asymptotic expansion of the generating function and that of its coefficients. See \cite[Sec.~5]{FlajoletSedgewick2009} for a detailed discussion about this comparison.

In this paper, we present a generalization of the transfer theorem to the multivariate setting. Below, we start by recalling some basic notations about multivariate generating functions, and then define the regimes of coefficient asymptotics with which our transfer theorem will be concerned.

\paragraph{Multivariate generating functions.}
In many practical problems, the relevant information is  captured naturally by a multidimensional infinite array $(a_{\bm n})_{\bm n \in \natural^d} \in \complex^{\natural^d}$.
Such multidimensional arrays and the multivariate generating functions which encode them will be this paper's central objects.
To make the formulas compact, we will use the following multi-index notations: For any formal or complex vectors $\bm z=(z_1,\ldots,z_d)$ and $\bm \theta=(\theta_1,\ldots,\theta_d)$, we denote
\begin{equation}\label{eq:multi-index notation vector}
\bm \theta \bm z = (\theta_1 z_1, \ldots, \theta_d z_d)
\,,\qquad
\bm{\theta \cdot z} = \theta_1 z_1 + \cdots + \theta_d z_d
\,,\qquad
\bm z^{\bm \theta} = z_1^{\theta_1} \cdots z_d^{\theta_d}
\,,\qquad
\dd \bm z = \dd z_1 \cdots \dd z_d \,.
\end{equation}
And, for any scalar $\sigma$ and
integer vector $\bm m \in \natural^d$, let
\begin{equation}\label{eq:multi-index notation scalar}
\sigma^{\bm \theta} = (\sigma^{\theta_1},\ldots,\sigma^{\theta_d})
\qtq{and}
\bm{\partial^m} = \partial_1^{m_1}\cdots \partial_d^{m_d}
\equiv \pd[m_1]{}{\!z_1\!}\cdots \pd[m_d]{}{\!z_d}.
\end{equation}
With these notations, the multivariate generating function of the array $(a_{\bm n})_{\bm n\in \natural^d}$ can be written as
\begin{equation}\label{eq:multiv GF}
A(\bm z) = \sum_{\bm n \in \natural^d} a_{\bm n} \bm{z^n}
\equiv \sum_{n_1 = 0}^\infty \cdots \sum_{n_d = 0}^\infty a_{n_1,\ldots,n_d}\, z_1^{n_1} \cdots z_d^{n_d}.
\end{equation}
Similarly to the univariate case, we denote the coefficient  $a_{\bm n}$ by $[\bm{z^n}] A(\bm z)$.
We assume that every generating function in this paper is absolutely convergent in an open neighborhood of $\bm 0\equiv (0,\ldots,0)\in \complex^d$, so that it defines an analytic function there.


We refer to \cite{Hormander1990} for the general theory on power series and analytic functions in several variables.
One particular fact that we will use without further mention is the uniqueness of analytic continuation: if $A$ and $B$ are analytic functions on an open connected domain $\Omega\subseteq \complex^d$ and $A(\bm z)=B(\bm z)$ on any open subset of $\Omega$, then $A(\bm z)=B(\bm z)$ for all $\bm z\in \Omega$.

\paragraph{Stretched diagonal limits.}
One central problem of analytic combinatorics in several variables is to understand the asymptotics of the coefficients $[\bm{z^n}]A(\bm z)$ when the components $n_1,\ldots,n_d$ of the multi-index tend to $\infty$ \emph{simultaneously}.
In general, one needs to put some constraint on the relative speeds at which $n_1,\ldots,n_d$ grow in order to get a useful asymptotic formula, and the interesting regimes are to a large extent dictated by the structure of the singularities of $A(\bm z)$.

In this work, we will be interested in the \emph{stretched diagonal limits}, where $n_1,\ldots,n_d$ grow at polynomial speeds relative to each other. In other words, we require that for each $j=2,\ldots,d$, there exists a constant $\theta_j > 0$ such that the ratio $n_j /n_1^{\theta_j}$ remains in some compact interval $\mathcal I \subset \real_{>0} \equiv (0,\infty)$ when $n_1 \to \infty$.
A symmetrized definition of this limit regime goes as follows:
\begin{definition*}[Stretched diagonal limit]\label{def:stretched diagonal limit}
We say that the multi-index $\bm n \in \natural^d$ tends to $\bm \infty$ in the \emph{stretched diagonal regime} if there exist an exponent vector $\bm \theta \in \real_{>0}^d$ and an auxiliary variable $n_0>0$, such that
\begin{equation}
\bm n = \bm \lambda n_0^{\bm \theta} \equiv \m({ \lambda_1 n_0^{\theta_1}, \ldots, \lambda_d n_0^{\theta_d} }
\end{equation}
for some prefactor $\bm \lambda = \bm \lambda(n_0)$ that remains in a compact set $\mathcal K\subset \real_{>0}^d$ when $n_0\to \infty$.
In this case, we will also say that $\bm n\to \bm \infty$ in the \emph{\diag\ limit}.
\end{definition*}

\begin{remark*}


If we require in addition that $\bm \lambda(n_0) \to \bm \lambda_* \in \real_{>0}^d$ as $n_0\to \infty$, then the point $\bm n =\bm \lambda n_0^{\bm \theta}$ would tend to $\bm \infty$ roughly along the curve $\mathcal D_{ \bm \lambda_*,\bm \theta} = (\bm \lambda_* t^{\bm \theta},t>0)$. But we only require $\bm \lambda(n_0)$ to stay in some compact set. Intuitively this means that $\bm n$ can jump between the curves $\mathcal D_{\bm \lambda_*,\bm \theta}$ for different values of~$\bm \lambda_*$.
Consequently, the asymptotics that we write in the stretched diagonal limit should be understood as uniform \wrt\ $\bm \lambda_*$ on every compact subset of $\real_{>0}^d$.

Notice that for any $\tau>0$, the $(\tau \bm \theta)$-diagonal limit is the same as the \diag\ limit: it suffices to replace the variable $n_0$ by $n_0^{\tau}$ or $n_0^{1/\tau}$ to pass from one to the other. 
The classical notion of \emph{diagonal limit}, e.g.\ as defined in \cite{PemantleWilson2013}, corresponds to the $(1,\cdots,1)$-diagonal limit in our terminology.
\end{remark*}

Apart from being the relevant limit regime for multivariate $\Delta$-analytic function with power-law singularity (to be defined below), the stretch diagonal limits also arise naturally from the study of critical phenomena in probability and mathematical physics.

\paragraph{Preliminary definitions.}
Let us define some domains and function classes needed for stating the main theorems.
Let $\complex_*=\complex\setminus\{0\}$ and $\real_{>0}=(0,\infty)$.
For $\delta>0$, we write
$K_\delta = \Set{u \in \complex_*}{|\arg(u)|<\delta}$ and denote $\Omega_\delta = K_{\pi/2+\delta}$. The multivariate versions of these cones are denoted $\kdom = K_\delta^d$ and $\odom = \Omega_\delta^d$.
For $\epsilon>0$, $w\in \complex$ and $\bm w \in \complex^d$, let $B_\epsilon(w) = \Set{z\in \complex}{|z-w|<\epsilon}$ and $\bm B_\epsilon(\bm w) = B_\epsilon(w_1) \times \cdots \times B_\epsilon(w_d)$.
Notice that $\Omega_\delta$ is related to the $\Delta$-domain $\dom$ by $\Omega_\delta = \set{\mu(1-z)}{\mu>0 \text{ and } z\in \dom}$.

\begin{definition*}[Multivariate $\Delta$-analytic functions]
Let $\bm \rho \in \complex_*^d$. We say that a multivariate function $A(\bm z)$ is \emph{$\Delta$-analytic at $\bm \rho$} if it has an analytic continuation on the product domain $\bm \rho {\ddom}:= (\rho_1 \dom) \times \cdots \times (\rho_d \dom)$ for some $\delta>0$, where $\dom$ is the univariate $\Delta$-domain defined in \eqref{eq:def Delta domain}.
\end{definition*}

Like in the univariate case, one can make the change of variable $\bm z'=\bm z/\bm \rho \equiv (z_1/\rho_1,\ldots, z_d/\rho_d)$ to bring the point $\bm \rho$ to $\bm 1\equiv (1,\ldots,1)$. In the following, we will focus \wlg\ on functions which are $\Delta$-analytic at $\bm 1$.

\begin{definition*}[Demi-analytic functions]
Let $\bm \Omega=\Omega_1 \times \cdots \times \Omega_d$ be an open product domain in $\complex^d$ and let $\bm \rho \in (\partial \Omega_1)\times \cdots \times (\partial \Omega_d)$. For $j\in \{1,\ldots,d\}$, denote $\bm \Omega_{\hat j} = \Omega_1 \times \cdots \times \Omega_{j-1} \times \complex \times \Omega_{j+1} \times \cdots \times \Omega_d$.
We say that a function $A:\bm \Omega \to \complex$ is \emph{demi-analytic at $\bm \rho \in \bm \Omega$} if $A$ is analytic on $\bm \Omega$ and there exist $\epsilon>0$ and a decomposition $A=A_1+\cdots +A_d$ such that $A_j$ is analytic in $B_\epsilon(\bm \rho) \cap \bm \Omega_{\hat j}$ for each $j\in \{1,\ldots,d\}$.
If each term $A_j$ in the above decomposition is analytic on $\bm \Omega_{\hat j}$, then we say that $A$ is \emph{demi-entire}.
\end{definition*}

\begin{definition*}[Generalized homogeneous functions]
Let $\bm K$ be a cone in $\complex^d$, i.e.\ a subset of $\complex^d$ such that $\set{\sigma \bm z}{\bm z\in \bm K} = \bm K$ for all $\sigma>0$.
For $\theta_0 \in \real$ and $\bm \theta \in \real_{>0}^d$, we say that a function $H:\bm K\to \complex$ is \emph{\thetahomo} if
\begin{equation}\label{eq:def homogeneous}
H(\sigma^{\bm \theta}\bm u) \equiv H(\sigma^{\theta_1}u_1,\ldots, \sigma^{\theta_d}u_d)= \sigma^{\theta_0}H(\bm u)
\end{equation}
for all $\bm u\in \bm K$ and $\sigma>0$.
\end{definition*}

It is clear that a \thetahomo\ function is also $(\tau \theta_0,\tau \bm \theta)$-homogeneous for all $\tau>0$.
The classical notion of \emph{homogeneous functions of degree $D$} becomes \emph{$(D,\bm 1)$-homogeneous} in our terminology.

The three definitions above generalize respectively the notions of $\Delta$-analytic functions, locally analytic functions (for the term $A\1{reg}$ in \Cref{thm:u-transfer bis}), and power functions (for the term $h\cdot (1-z)^{\alpha}$) used in the univariate transfer theorems.
To state the multivariate transfer theorem, we need one more definition whose counterpart does not appear explicitly in the univariate setting:

\begin{definition*}[Functions of polynomial type]
For any $\delta>0$,
we say that a function $F$ is \emph{of polynomial type (globally) on $\kdom$} if there exist $C,M>0$ such that
\begin{equation}\label{eq:def polynomial type 2-ended}
\forall \bm u\in \kdom,\quad
\abs{F(\bm u)} \le C\cdot \m({ |u_1|^{-M} + |u_1|^M + \cdots + |u_d|^{-M} + |u_d|^M }
\end{equation}
We say that $F$ is \emph{of polynomial type locally at $\bm 0\in \kdom$} if the above bound only holds in a neighborhood of $\bm 0\in \kdom$. Equivalently, $F$ is \emph{of polynomial type locally at $\bm 0\in \kdom$} if there exist $\epsilon,C,M>0$ such that
\begin{equation}\label{eq:def polynomial type near 0}
\forall \bm u\in \kdom \cap \bm B_{\bm 0,\epsilon}\,,\quad
\abs{F(\bm u)} \le C\cdot \m({ |u_1|^{-M} + \cdots + |u_d|^{-M}}.
\end{equation}
Similarly, we say that a function $A$ is \emph{of polynomial type on $\ddom$} if there exist $C,M>0$ such that
\begin{equation}\label{eq:def polynomial type near 1}
\forall \bm z\in \ddom,\quad
\abs{A(\bm z)} \le C\cdot \m({ |z_1-1|^{-M} + \cdots + |z_d-1|^{-M} }
\end{equation}
(here we do not need the terms $|z_j-1|^M$ on the \rhs\ because $\ddom$ is bounded),
and we say that $A$ is \emph{of polynomial type locally at $\bm 1 \in \ddom$} if it satisfies the above bound in $\ddom \cap \bm B_{\bm 1,\epsilon}$ for some $\epsilon>0$.
For $\bm S = \kdom$, $\odom$ or $\ddom$, we write $\mathcal P(\bm S) =\set{f:\bm S\to \complex}{f \text{ is analytic and of polynomial type on }\bm S}$.
\end{definition*}

Up to decreasing $\delta$, a continuous function on $\ddom$ is always bounded by a constant on $(\dom\setminus U)^d$ for any neighborhood $U$ of $1\in \complex$. Thus, the condition that $A$ is of polynomial type (\emph{globally}) on $\ddom$ is essentially an upper bound for $\abs{A(\bm z)}$ when $z_j\to 1$ in $\dom$ for \emph{some} $j\in \{1,\ldots,d\}$.
In contrast, the condition of being of polynomial type \emph{locally} at $\bm 1\in \ddom$ only gives a bound for $|A(\bm z)|$ when $z_j\to 1$ for \emph{all} $j\in \{1,\ldots,d\}$.
When $d=1$, the two conditions are essentially the same. They do not appear explicitly in the statement of the univariarte transfer theorem because a function having an asymptotic expansion of the form \eqref{eq:u-transfer fun expansion} is automatically of polynomial type on $\dom$.

\paragraph{Main results.}
We are now ready to state the multivariate transfer theorem. Like in the univariate case, we formulate it in two ways, one discussing the building blocks of the coefficient asymptotics piece by piece, and the other showing how the theorem would be applied in practice.

\begin{theorem}[Multivariate transfer theorem]\label{thm:m-transfer}
Let $\theta_0\in \real$, $\bm \theta\in \real_{>0}^d$ and $A\in \poly(\ddom)$ for some $\delta>0$.
We consider the asymptotics of the coefficients $[\bm{z^n}]A(\bm z)$ when $\bm n\to \bm \infty$ in the \diag\ regime.
\begin{enumerate}
\item (Demi-analytic terms have exponentially small contribution)\\
If $A$ is demi-analytic at $\bm 1\in \ddom$, then there exist $r\in (0,1)$ and $\sigma>0$ such that $[\bm{z^n}]A(\bm z) = O(r^{n_0^\sigma})$.

\item (Coefficient asymptotics of generalized homogeneous functions)\\
When $A(\bm z) = H(\bm 1-\bm z)$ and $H$ is \thetahomo\ ($\theta_0\in \real$), we have the asymptotic expansion
\begin{equation}\label{eq:m-transfer power expansion}
[\bm{z^n}] H(\bm 1-\bm z) = \frac1{n_0^\Theta}
\sum_{\bm k \in \natural^d} \frac{
 D_{\bm k} I(\bm \lambda)}{ n_0^{\bm{k\cdot \theta}} }
\equiv
\frac1{n_0^\Theta} \m({ I(\bm \lambda) + \sum_{j=1}^d \frac{\partial_j I(\bm \lambda) + \frac12 \lambda_j \partial_j^2 I(\bm \lambda)}{n_0^{\theta_j}} + \cdots }.
\end{equation}
where $\Theta = \theta_0+\theta_1+\cdots+\theta_d$, $D_{\bm k}=\sum_{\bm l\le \bm k} \, g_{\bm k,\bm l} \cdot \bm{\lambda^l} \bm \partial^{\bm k+\bm l}$ is a partial differential operator of order $2(k_1+\cdots+k_d)$, and $I:\real_{>0}^d \to \complex$ is the inverse Laplace transform of $H$ defined by
\begin{equation}\label{eq:m-transfer scaling function}
I(\bm \lambda) = \m({ \frac{1}{2\pi i} }^d \int_{\vv_{\delta'}} H(\bm u) e^{\bm {\lambda \cdot u}} \dd \bm u \,.
\end{equation}
In the above formulas, we denote $g_{\bm k,\bm l} = g_{k_1,l_1}\cdots g_{k_d,l_d}$, with the numbers $(g_{k,l})_{k,l\ge 0}$ defined by \eqref{eq:u-transfer g-factors}. The sum $\sum_{\bm l\le \bm k}$ runs over $\set{\bm l\in \natural^d}{l_j\le k_j \text{ for all }1\le j\le d}$, while the integral is over~$\bm V_{\delta'} = V_{\delta'}^d$, where $V_{\delta'} \subset \Omega_\delta$ is any piecewise smooth curve which coincides with the rays $\partial \Omega_{\delta'}$ outside a bounded region for some $\delta'\in (0,\delta)$, e.g.\ as in Figure~\refp{b}{fig:Delta/cone contour definition}.
(Recall that $\Omega_\delta = \Set{u\in \complex_*}{|\arg(u)|<\frac\pi2 + \delta}$.)

\item (Big-O transfer and little-o transfer)\\
If $A(\bm z)=O(\tilde H(\bm 1-\bm z))$ for some \thetahomo\ $\tilde H$ as $\bm z\to \bm 1$ in $\ddom$, then $[\bm{z^n}] A(\bm z) = O(n_0^{-\Theta})$.\\
If $A(\bm z)=o\,(\tilde H(\bm 1-\bm z))$ for some \thetahomo\ $\tilde H$ as $\bm z\to \bm 1$ in $\ddom$, then $[\bm{z^n}] A(\bm z) = o\,(n_0^{-\Theta})$.
\end{enumerate}
\end{theorem}

\begin{theorembis}[Multivariate transfer theorem]{thm:m-transfer}
Let $\bm \theta\in \real_{>0}^d$ and $A \in \poly(\ddom)$ for some $\delta>0$. Assume that when $\bm z\to \bm 1$ in $\ddom$, the function $A$ has an asymptotic expansion of the form
\begin{equation}\label{eq:m-transfer fun expansion}
A(\bm z) = A\1{reg}(\bm z) + H_0(\bm 1-\bm z) + \cdots + H_{m-1}(\bm 1-\bm z) + o\m({H_m(\bm 1-\bm z)}\,,
\end{equation}
where $A\1{reg}$ is demi-analytic at $\bm 1\in \ddom$, and each $H_k$ is \thetahomo[\theta_0\0k] and of polynomial type locally at $\bm 0\in \odom$, with $\theta_0\00>\cdots>\theta_0\0{m-1}\ge \theta_0\0m$.
Then as $\bm n\to \bm \infty$ in the \diag\ limit, we~have
\begin{equation}\label{eq:m-transfer coeff expansion}
[\bm{z^n}]A(\bm z) = [\bm{z^n}] H_0(\bm 1-\bm z) + \cdots + [\bm{z^n}] H_{m-1}(\bm 1-\bm z) + o(n_0^{-\Theta_m})
\end{equation}
where each $[\bm{z^n}] H_k(\bm 1-\bm z)$ has the asymptotic expansion given in Theorem~\refp2{thm:m-transfer}, $\Theta_m =\theta_0\0m +\theta_1 +\cdots +\theta_d$, and the little-o estimate is uniform \wrt\ $\bm \lambda$ on all compact subsets of $\real_{>0}^d$.
\\
The same result holds if the little-o estimates in both \eqref{eq:m-transfer fun expansion} and \eqref{eq:m-transfer coeff expansion} are replaced by big-O estimates.
\end{theorembis}

\begin{remark*}~The following remarks are analoguous to their univariate counterparts below \Cref{thm:u-transfer bis}.
\begin{enumerate}[nolistsep,topsep=0.5ex,itemsep=1ex]
\item
Theorems~\ref{thm:m-transfer}~and~\ref{thm:m-transfer bis} are not exactly equivalent, for the same reason as in the univariate case.

\item
The asymptotic expansion \eqref{eq:m-transfer power expansion} in Theorem~\refp2{thm:m-transfer} uses the same modified notion of asymptotic expansion as in Theorem~\refp2{thm:u-transfer}, with the asymptotic scale $s_\beta(n_0)=n_0^{-\beta}$ ($\beta \in \real$).

\item
It is also possible to include $\theta_0\in \complex \setminus \real$ in \Cref{thm:m-transfer,thm:m-transfer bis}. The expansion \eqref{eq:m-transfer power expansion} would still hold with a complex value for $\Theta$. But we restrict ourselves to $\theta_0\in \real$ for simplicity.

\item
Theorem~\refp{1)--(2}{thm:m-transfer} imply that for a \thetahomo\ function $H\in \poly(\odom)$, if $H$ is demi-analytic at $\bm 0\in \odom$, then its inverse Laplace transform $I(\bm \lambda)$ vanishes for all $\bm \lambda \in \real_{>0}^d$. We will see in \Cref{cor:scaling function properties} below that the converse is also true.
\end{enumerate}
Let us also make some remarks about the expansion \eqref{eq:m-transfer fun expansion} in Theorem~\ref{thm:m-transfer bis}. These will be discussed in more details in \Cref{sec:discussions}.

\begin{enumerate}[resume,nolistsep,topsep=0.5ex,itemsep=1ex]
\item
The expansion \eqref{eq:m-transfer fun expansion} is in general not unique. The reason is that a demi-analytic function can also contain \thetahomo\ components for any $\theta_0 \in \real$. In principle, one could fix $A\1{reg}=0$ in Theorem~\ref{thm:m-transfer bis} without reducing significantly the class of functions $A(\bm z)$ covered by the theorem. But having the flexibility of choosing any demi-analytic function $A\1{reg}$ makes the theorem easier to apply.

\item
Once $A\1{reg}$ is chosen, one can write $A(\bm z) = A\1{reg}(\bm z) + A\1{sing}(\bm 1-\bm z)$. If the expansion \eqref{eq:m-transfer fun expansion} exists, then its terms $H_k$ can be obtained as the coefficients in the \emph{univariate} asymptotic expansion $A\1{sing}(\varepsilon^{\bm \theta} \bm u) = H_0(\bm u) \cdot \varepsilon^{\theta_0\00} + \cdots + H_{m-1}(\bm u) \cdot \varepsilon^{\theta_0\0{m-1}} + o(\varepsilon^{\theta_0\0m})$ \wrt\ $\varepsilon \to 0^+$.

\item
Although not assumed in Theorem~\ref{thm:m-transfer bis}, the functions $\bm z\mapsto H_k(\bm 1-\bm z)$ ($0\le k<m$) are necessarily analytic on $\ddom$. See \Cref{lem:H_k analytic} below. In particular, the coefficients $\an[H_k(\bm 1-\bm z)]$ are well-defined.

\item
Unlike the univariate case, the generalized homogenous function $H_m$ used in the little-o estimate is in general different from the previous term $H_{m-1}$ of the asymptotic expansion. The reason is that the class of \thetahomo\ functions is one-dimensional (generated by $H(u)=u^{\theta_0/\theta})$ in the univariate case, but infinite-dimensional in the multivariate case.
\end{enumerate}
\end{remark*}

In many applications, one is only interested in the dominant asymptotics of the coefficients. For this, we can simplfy Theorems~\ref{thm:m-transfer}~and~\ref{thm:m-transfer bis} to the following statement: for any generating function $A \in \poly(\ddom)$,
\begin{equation}
A(\bm z) = A\1{reg}(\bm z) + H(\bm 1-\bm z) + o(\tilde H(\bm 1-\bm z))
\qquad \Rightarrow \qquad
[\bm{z^n}]A(\bm z) \sim I(\bm \lambda) \cdot n_0^{-\Theta}
\end{equation}
where $A\1{reg}$ is demi-analytic and of polynomial type locally at $\bm 1\in \ddom$, $H$ and $\tilde H$ are \thetahomo\ and of polynomial type locally at $\bm 0\in \ddom$, and $\Theta\in \real$ and $I:\real_{>0} \to \complex$ are defined as in Theorem~\refp2{thm:m-transfer}. As in the theorems, the asymptotics of the coefficients is taken in the \diag\ regime.

The functional form of the prefactor $I(\bm \lambda)$ is often of practical importance (see \Cref{sec:discussions} for more discussions). In Theorem~\ref{thm:m-transfer}, $I(\bm \lambda)$ was expressed as an integral transform of $H(\bm u)$. The next theorem will provide some basic properties of this transform and its inverse.
We define the \emph{inverse Laplace transform} (also called \emph{Borel transform}) of a function $H$ by
\begin{equation}\label{eq:def Borel transform}
\borel[H](\bm \lambda) = \pid \int_{\vv_{\delta'}} e^{\bm { \lambda \cdot u }} H(\bm u) \, \dd \bm u
\end{equation}
where $\delta'\in (0,\delta)$ and $\vv_{\delta'}$ is a contour of the form specified below \Cref{eq:m-transfer scaling function}.
On the other hand, for any given $\bm c\in \real_{>0}^d$, we define the \emph{Laplace tranform (truncated at $\bm c$)} of a function $I$ by
\begin{equation}
\laplace{}[I](\bm u)
= \int_{\bm c}^{\bm \infty} \!
e^{-\bm{ \lambda \cdot u}} I(\bm \lambda) \,\dd \bm \lambda
\equiv \int_{[c_1,\infty)\times \cdots \times [c_d,\infty)}
e^{-\bm{ \lambda \cdot u}} I(\bm \lambda) \,\dd \bm \lambda.
\end{equation}

\begin{theorem}[Properties of the Borel-Laplace transforms]
\label{thm:Borel-Laplace}
Fix $\delta\in (0,\pi/2)$ and $\bm c\in \real_{>0}^d$.
\begin{enumerate}[itemsep=0ex,topsep=0.5ex]
\item
For all $H\in \poly(\odom)$, $\borel[H]$ defines an analytic function on $\kdom$ independent of the value  of $\delta'\in (0,\delta)$. Moreover, $\borel[H] \in \poly(\kdom[\delta^\circ])$ for all $\delta^\circ\in (0,\delta)$.

\item
For all $I\in \poly(\kdom)$, $\laplace{}[I]$ defines an analytic function on $\odom$.
Moreover, for all $\delta^\circ \in (0,\delta)$, there exists $M>0$ such that $\borel \circ \laplace{} [I]$ is well-defined and analytic on $\Set{\bm \lambda \in \kdom[\delta^\circ]}{\forall j,\, |\lambda_j|>M}$.

\item ($\laplace$ is a right inverse of $\borel$).
For all $I\in \poly(\kdom)$, we have $\borel \circ \laplace{}[I]=I$.\\
For all $H\in \poly(\odom)$, there exists a demi-entire function $E_{\bm c}:\odom\to \complex$ such that $\mathcal L_{\bm c}\circ \borel[H]=H+E_{\bm c}$.
\end{enumerate}
\end{theorem}

\begin{corollary}\label{cor:scaling function properties}
The scaling function $I(\bm \lambda)$ in Theorem~\refp2{thm:m-transfer} is \thetahomo[-\Theta] and in $\mathcal P(\kdom[\delta^\circ])$ for all $\delta^\circ \in (0,\delta)$. It is identically zero \Iff\ $H$ is demi-entire on $\odom$.
\end{corollary}

\paragraph{Outline of the rest of the paper.}
\Cref{sec:main proof,sec:Borel-Laplace} give the proofs of \Cref{thm:m-transfer bis,thm:Borel-Laplace}, respectively.
In \Cref{sec:discussions}, I first provide some additional results on the classes of multivariate functions mentioned above \Cref{thm:m-transfer}, and then discuss the background of this paper and its relations to previous works.

\paragraph{Acknowledgement.}
The author is grateful for the support of the ETH Foundation.
This work has also been supported by the Swiss National Science Foundation (SNF) Grant 175505 and the Agence Nationale de la Recherche project ProGraM (Projet-ANR-19-CE40-0025).

\begin{figure}
\centering
\includegraphics[scale=1,page=1]{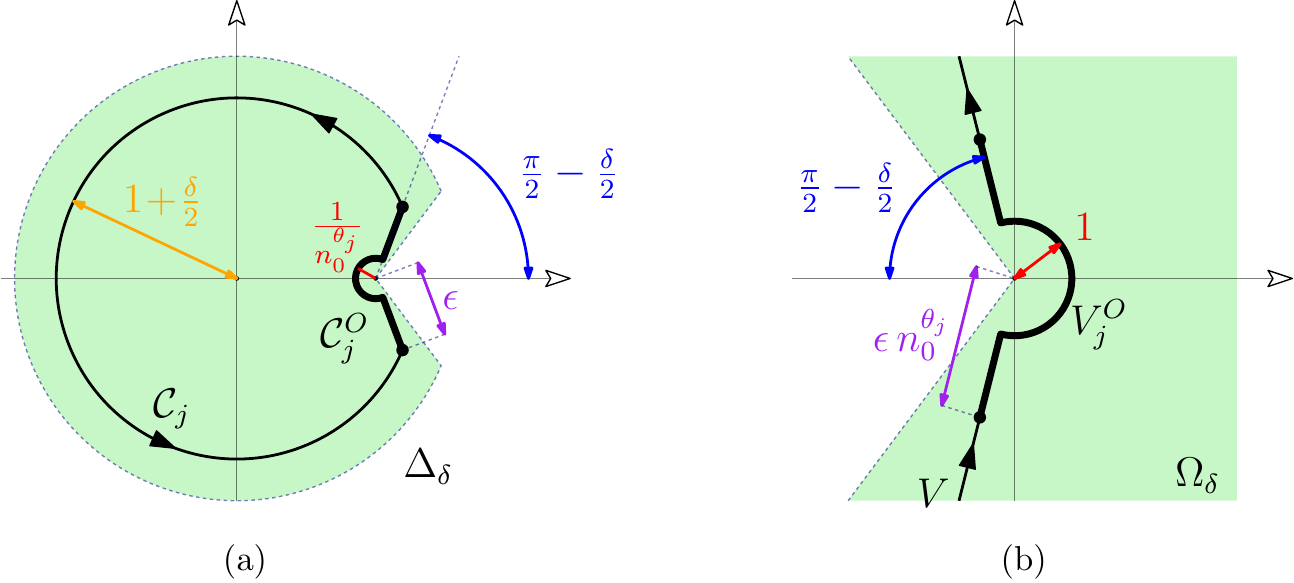}
\caption{(a) The contour $\cj$ is composed of two circle arcs and two line segments. The two arcs have their centers at $0$ and $1$, and their radii $1+\delta/2$ and $1/\nj$, respectively. The line segments connect the two arcs and form an angle $\frac\pi2-\frac\delta2$ with the positive real axis at $1$. The contour~$\cj^O$ is $\cj$ with the large arc removed.
(b) The contour $\vj$ is the image of $\mathcal C_j^O$ under the mapping $z \mapsto \nj(1-z)$, and $V$ is the limit of $\vj$ as $n_0\to \infty$ (which no longer depends on~$j$). $V$ is a special case of the curve $V_{\delta'}$ appeared in the definition \eqref{eq:m-transfer scaling function} of $I(\bm \lambda)$ in \Cref{thm:m-transfer}.
}
\label{fig:Delta/cone contour definition}
\end{figure}

\section{Proof of the multivariate transfer theorems}\label{sec:main proof}

\paragraph{Contours and domains.}
We start by defining some contours and domains useful in the proof. Recall the definition~\eqref{eq:def Delta domain} of the $\Delta$-domain $\dom$. For given values of $n_0\gg 1$ and $\theta_j>0$, we define the contours $\cj$ and $\cj^O$ by Figure~\refp{a}{fig:Delta/cone contour definition}: $\cj$ is a contour inside $\ddom$ and close to its boundary, while $\cj^O$ is a portion of~$\cj$ close to the point $z=1$. We denote by $\epsilon=\epsilon(\delta)$ the maximal distance from a point of $\cj^O$ to $1$.
Let $\vj$ be the image of $\cj^O$ under the mapping $z\mapsto \nj(1-z)$, and $V$ be the limit of $\vj$ as $n_0\to \infty$.
We have $\vj \subset V \subset \Omega_\delta$, see Figure~\refp{b}{fig:Delta/cone contour definition}.
The multivariate versions of these objects are defined by
\begin{align*}
  \cc   &= \cj[1] \times \cdots \times \cj[d]
& \cc^O &= \cj[1]^O \times \cdots \times \cj[d]^O \\
  \vv   &= V \times \cdots \times V
& \vv^O &= \vj[1] \times \cdots \times \vj[d] \,.
\end{align*}

\paragraph{Uniform growth bounds.}
Now let us derive some uniform bounds of functions on the above contours and domains.
Recall that in the \diag\ limit, the vector $\bm \lambda = \nni \bm n \equiv (\nji[1]n_1,\ldots,\nji[d]n_d)$ remains in some compact subset of $(0,\infty)^d$ as $n_0\to \infty$.
Let $\lambda\1{min}, \lambda\1{max} >0$ be such that $\lambda_j\in [\lambda\1{min},\lambda\1{max}]$ for all $1\le j\le d$.
Let $\theta\1{min} = \min \{\theta_1,\ldots,\theta_d\}$ and $\theta\1{max} = \max \{\theta_1,\ldots,\theta_d\}$. We assume $n_0\ge 2$.

\newcommand{\bb}[1][1]{\bm{B}_{\bm{#1},\epsilon}}
We assume \wlg\ that $\delta$, and therefore $\epsilon$, is small enough so that all the local assumptions near $\bm z=\bm 1$ in Theorem~\ref{thm:m-transfer bis} hold globally in $\bb := \set{\bm z\in \complex^d}{\forall 1\le j\le d ,\, |z_j-1|<\epsilon}$. In other words, the functions $A\1{reg}$ and $\bm z \mapsto H_k(\bm 1-\bm z)$ are well-defined on $\ddom \cap \bb$ and satisfy
\begin{equation}\label{eq:main proof F(z) domain bound}
\forall \bm z\in \ddom \cap \bb ,\quad
|F(\bm z)| \le C \cdot \m({ |z_1-1|^{-M} + \cdots + |z_d-1|^{-M} } \,,
\end{equation}
and there exist functions $A\1{reg,\mathnormal j}$ analytic in $(\dom^{j-1} \times \complex \times \dom^{d-j}) \cap \bb$ such that $A\1{reg} = A\1{reg,\mathrm 1} + \cdots + A\1{reg,\mathnormal d}$. Notice that $\cc^O \subset \ddom \cap \bb$ and $|z_j-1|\ge \nji$ for all $\bm z\in \cc$. So the bound \eqref{eq:main proof F(z) domain bound} implies that
\begin{equation}\label{eq:main proof F(z) contour bound}
\forall \bm z\in \cc^O ,\quad
|F(\bm z)| \le C\cdot \m({n_0^{M \theta_1} + \cdots + n_0^{M \theta_d}}
\le dC \cdot n_0^{M \theta\1{max}} .
\end{equation}
Since $A$ is of polynomial type \emph{globally} on $\ddom$, it satisfies \eqref{eq:main proof F(z) domain bound} for all $\bm z\in \ddom$ and \eqref{eq:main proof F(z) contour bound} for all~$\bm z\in \cc$.
We leave the reader to check the elementary fact that there exist $c,\mu>0$ depending only on $\delta$, $\theta\1{min}$ and $\lambda\1{min}$, $\lambda\1{max}$, such that
\begin{equation}\label{eq:main proof z^n bound}
\forall z\in \cj^O ,\quad
\mb|{z^{n_j+1}} \ge c \cdot \exp \mb({\mu \cdot \nj \Re(z-1)}  \,.
\end{equation}
And since $\Re(z-1) \ge \nji\!$ for $z\in \cj^O$ and $|z|>1$ for $z\in \cj \setminus \cj^O$, there exists $\tilde c=\tilde c(\delta,\theta\1{min},\lambda\1{min},\lambda\1{max})>0$ such that
\begin{equation}\label{eq:main proof z^n constant bound}
\forall z \in \cj ,\quad
\mb|{z^{n_j+1}}\ge \tilde c \,.
\end{equation}
Under the change of variable $u=\nj(1-z)$, the bound \eqref{eq:main proof z^n bound} becomes
\begin{equation}\label{eq:main proof (1-u)^n bound}
\forall u \in \vj ,\quad
\mb|{ (1-\nji u)^{n_j+1} } \ge c \cdot e^{- \mu \cdot \Re(u) } \,.
\end{equation}
In \Cref{sec:discussions}, we will discuss several equivalent formulations of the polynomial-type condition, and their consequences on homogeneous functions. In particular, we will see in \Cref{lem:homogeneous monomial bounds} that a homogeneous function of polynomial type \emph{locally at $\bm 0\in \odom$} is always bounded by a Laurent polynomial \emph{globally~on~$\odom$}. An easy consequence of \Cref{lem:homogeneous monomial bounds} is the following global polynomial bound for the functions $H_0,\ldots, H_m$:
there exists $C,M>0$ such that
\begin{equation}\label{eq:main proof H(u) bound}
\forall \bm u \in \vv,\quad
|H_k(\bm u)| \le C\cdot |u_1\cdots u_d|^M
\end{equation}

\paragraph{Analyticity of the homogeneous components.}
Before starting the proof of the main theorems, let us show the analyticity of the functions $H_k$ ($0\le k<m$) mentioned in the 7-th remark below~Theorem~\ref{thm:m-transfer bis}. Actually, we show that they are analytic on the larger domain $\odom$, thanks to homogeneity.

\begin{lemma}\label{lem:H_k analytic}
The functions $H_0,\ldots,H_{m-1}$ in  Theorem~\ref{thm:m-transfer bis} have analytic continuations on $\odom$.
\end{lemma}

\begin{proof}
Recall that by our assumption on the smallness of $\epsilon$, the function $A\1{reg}$ is analytic on $\ddom \cap \bb$.
Let $\bb[0]=\setn{\bm 1-\bm z}{\bm z \in \bb} \equiv \setn{\bm z\in \complex^d}{\forall 1\le j\le d,\, |z_j|<\epsilon}$ and $A\1{sing}(\bm 1-\bm z)=A(\bm z)-A\1{reg}(\bm z)$ and $f_\sigma(\bm u) = \sigma^{-\theta_0\00} A\1{sing}(\sigma^{\bm \theta} \bm u)$. Then $f_\sigma$ is analytic on $\odom \cap \bb[0]$ for all $0<\sigma\le 1$, and the expansion \eqref{eq:m-transfer fun expansion} implies that $H_0(\bm u) = \lim_{\sigma\to 0^+} f_\sigma(\bm u)$ for all $\bm u \in \odom \cap \bb[0]$.
Moreover, thanks to the polynomial-type bound, $H_0,\ldots,H_m$ are bounded on all compact subsets of $\odom \cap \bb[0]$. It follows that the family $(f_\sigma)_{\sigma\in (0,1]}$ is uniformly bounded there.
By Vitali's theorem, the convergence $f_\sigma(\bm u) \cvg{}{\sigma \to 0} H_0(\bm u)$ is uniform on compact subsets of $\odom \cap \bb[0]$. Therefore $H_0$ is analytic on $\odom \cap \bb[0]$.

The same argument applied  to $A\1{sing}\!-H_0$, $A\1{sing}\!-H_0-H_1$, etc.~shows that $H_1,\ldots,H_{m-1}$ are~also analytic on $\odom \cap \bb[0]$. The homogeneity property $H_k(\sigma^{\bm \theta} \bm u)=\sigma^{\theta_0\0k}H_k(\bm u)$ extends this to all $\bm u \in \odom$.
\end{proof}

\paragraph{Proof of Theorem~\ref{thm:m-transfer bis}.}
Now let us prove Theorem~\ref{thm:m-transfer bis}. The proof will also imply \Cref{thm:m-transfer} by considering cases in which only one term on the \rhs\ of \eqref{eq:m-transfer fun expansion} is nonzero.
We follow the same general steps as the proof of the univariate transfer theorem in \cite{FlajoletOdlyzko1990}.

\vspace{-1em}
\paragraph{Step~1.~Contour deformation and localization.}
The coefficient $\an$ are related to $A(\bm z)$ by the Cauchy integral formula 
\begin{equation}\label{eq:Cauchy integral formula}
\an = \pid \oint_{\bm T_r} \frac{A(\bm z)}{\bm z^{\bm n+\bm 1}}\dd \bm z
\end{equation}
where $\bm T_r = \set{\bm z\in \complex^d}{\forall j,\, |z_j|=r}$ is the polytorus of radius $r$ for some $r>0$ small enough.
Thanks to the analyticity of $A(\bm z)$ in $\ddom$, we can deform the contour in \eqref{eq:Cauchy integral formula} from $\bm T_r$ to $\cc$. By spliting the integral according to the partition $\cc = \cc^O \cup (\cc \setminus \cc^O)$, we obtain
\begin{equation}
\an = I\1{loc} + \pid \int_{\cc \setminus \cc^O} \frac{A(\bm z)}{\bm z^{\bm n+\bm 1}}\dd \bm z
\qt{where}\quad
I\1{loc} = \pid \int_{\cc^O} \frac{A(\bm z)}{\bm z^{\bm n+\bm 1}}\dd \bm z \,.
\end{equation}
By the definition of $\cc^O$, for each $\bm z \in \cc \setminus \cc^O$, there is at least one $j^*\in \range d$ such that $|z_{j^*}|=1+\delta/2$. Together with the growth bounds \eqref{eq:main proof F(z) contour bound}~and~\eqref{eq:main proof z^n constant bound}, this implies
\begin{equation}\label{eq:main proof step 1 bound}
\forall \bm z\in \cc \setminus \cc^O,\quad
\abs{ \frac{A(\bm z)}{\bm z^{\bm n+\bm 1}} }
\le \frac{dC \cdot n_0^{M\theta\1{max}}}{
          \tilde c^{\,d-1}\cdot (1+\delta/2)^{n_{j^*}+1} }
\le \frac{\tilde C \cdot n_0^{M\theta\1{max}}}{
          (1+\delta/2)^{\lambda\1{min}\cdot n_0^{\theta\1{min}}} }
\le \mathtt C \cdot r^{n_0^{\theta\1{min}}}
\end{equation}
for some constants $r<1$ and $\mathtt C$ independent of $n_0$. It follows that
\begin{equation}\label{eq:main proof localization estimate}
\an = I\1{loc} + O \mb({ r^{n_0^{\theta\1{min}}} }
\end{equation}
for some constant $r\in (0,1)$ as $n_0\to\infty$.

\vspace{-1em}
\paragraph{Step~2.~Removing the demi-analytic term.}
Let $A\1{sing}(\bm 1-\bm z) = A(\bm z)-A\1{reg}(\bm z)$. Then we have
\begin{equation}
I\1{loc} = I\1{reg} + I\1{sing} :=
  \pid \int_{\cc^O} \frac{A\1{reg}(\bm z)}{\bm z^{\bm n+\bm 1}}\dd \bm z
+ \pid \int_{\cc^O} \frac{A\1{sing}(\bm 1-\bm z)}{\bm z^{\bm n+\bm 1}}\dd \bm z
\end{equation}
Recall that the demi-analytic term $A\1{reg}$ admits a decomposition $A\1{reg} = A\1{reg,\mathrm 1} + \cdots + A\1{reg,\mathnormal d}$, where each  $A\1{reg,\mathnormal j}$ is analytic in $(\dom^{j-1}\times \complex \times \dom^{d-j}) \cap \bb$. Within this domain, we can deform the $j$-th component the contour $\cc^O \equiv \cj[1]^O \times \cdots \times \cj[d]^O$ to the arc $\tilde{\mathcal C}_j^O$ on the circle $\Set{z\in \complex}{|z|=1+\delta/2}$ with the same endpoints as $\cj^O$. Then, with the same reasoning as in Step~1, we obtain
\begin{equation}
\int_{\cc^O} \frac{A\1{reg,\mathnormal j}(\bm z)}{\bm z^{\bm n+\bm 1}} \dd z
= \int_{\cj[1]^O\times \cdots \times \tilde{\mathcal C}_j^O \times \cdots \times \cj[d]^O} \frac{A\1{reg,\mathnormal j}(\bm z)}{\bm z^{\bm n+\bm 1}} \dd z
= O(r^{n_0^{\theta\1{min}}})
\end{equation}
for some $r\in (0,1)$. Summing over $j$ gives $I\1{reg} =  O(r^{n_0^{\theta\1{min}}})$. Together with Step~1, this implies
\begin{equation}
\an = I\1{sing} + O(r^{n_0^{\theta\1{min}}})
\end{equation}
for some constant $r\in (0,1)$ as $n_0\to \infty$.
In particular, when $A\1{sing}=0$, we obtain Case 1 of \Cref{thm:m-transfer}.

\vspace{-1em}
\paragraph{Step~3.~Transfer.}
By linearity, the expansion \eqref{eq:m-transfer fun expansion} implies $I\1{sing} = I_0 + \cdots +I_m$, where
\begin{equation}
I_k = \pid \!\!\int_{\cc^O}\!\! \frac{H_k(\bm 1-\bm z)}{\bm z^{\bm n+\bm 1}} \dd \bm z
\quad \text{for }0\le k<m
\qtq{and}
I_m = \pid \!\!\int_{\cc^O}\!\! \frac{o(H_m(\bm 1-\bm z))}{\bm z^{\bm n+\bm 1}} \dd \bm z
\end{equation}
Thanks to the analyticity of $H_0,\ldots,H_{m-1}$ in Lemma~\ref{lem:H_k analytic}, we can apply Step~1 to them to obtain
\begin{equation}
\an[H_k(\bm 1-\bm z)] = I_k + O(r^{n_0^{\theta\1{min}}}) \qt{for all }0\le k<m.
\end{equation}

Now let us show that $I_m=o(n_0^{-\Theta_m})$. To make use of the little-o estimate in the integrand, we want to further localize the contour $\cc^O$:
For a function $\varepsilon(x)\cv[]x 0$, we define the little-o versions of the contours $\cj^O$~and~$\cc^O$ by
\begin{equation}
\cj^o = \Setn{z\in \cj^O}{|z-1|<\varepsilon(n_0)}
\qtq{and}
\cc^o = \cj[1]^o \times \cdots \times \cj[d]^o \,.
\end{equation}
Contrary to $\cc^O$\!, which has a small but fixed size $\epsilon$, the contour $\cc^o$ shrinks to the point $\bm 1$ when $n_0\to \infty$.
According to the above definition, for all $\bm z \in \cc^O\setminus \cc^o$, there exists at least one $j^* \in \range d$ such that $|z_{j^*}| \ge 1+\sin(\delta/2)\cdot \varepsilon(n_0)$.
Then, similarly to \eqref{eq:main proof step 1 bound}, we have
\begin{equation}
\forall \cc^O \setminus \cc^o,\quad
\abs{\frac{H_m(\bm 1-\bm z)}{\bm z^{\bm n+\bm 1}}}
\le \frac{ \tilde C n_0^{M\theta\1{max}}
        }{ \mb({ 1+\sin(\delta/2) \cdot \varepsilon(n_0) }^{
           \lambda\1{min} \cdot n_0^{ \theta\1{min} } }}
\le \mathtt C \cdot
    r^{ \varepsilon(n_0) \cdot n_0^{\theta\1{min}} }
\end{equation}
for some $r<1$ and $\mathtt C>0$. We choose a function $\varepsilon(x)$ such that $\varepsilon(x)\gg x^{-\theta\1{min}}\log(x)$ when $x\to \infty$. (For example, $\varepsilon(x)=x^{\theta\1{min}/2}$.) Then the \rhs\ of the above display decays faster than any negative power of $n_0$ when $n_0\to \infty$. In particular, this implies
\begin{equation}\label{eq:m-transfer proof little-o localization}
\int_{\cc^O \setminus \cc^o} \frac{o(H_m(\bm 1-\bm z))}{\bm z^{\bm n+\bm 1}} \dd \bm z = o\m({ n_0^{-\Theta_m} }.
\end{equation}
as $n_0\to \infty$.
On the other hand, since the contour $\cc^o$ shrinks to $\bm z=\bm 1$ when $n_0\to \infty$, for any $\tau>0$, there exists $n_0^*$ such that
\begin{equation}
\forall \bm z\in \cc^o,\quad
\abs{o(H_m(\bm 1-\bm z))}\ \le\ \tau \abs{H_m(\bm 1-\bm z)}
\end{equation}
for all $n_0\ge n_0^*$. By plugging this inequality into the integral over $\cc^o$ and making the change of variable $\bm u=\nn(\bm 1-\bm z) \equiv (\nj[1](1-z_1),\ldots, \nj[d](1-z_d))$, we obtain
\begin{equation}\label{eq:m-transfer proof little-o transfer}
\abs{ \int_{\cc^o} \frac{o(H_m(\bm 1-\bm z))}{\bm z^{\bm n+\bm 1}} \dd \bm z }
\le \tau \int_{\vv^o} \m|{
    \frac{H_m(\nni \bm u)}{
          \mb({ \bm 1-\nni \bm u }^{\bm n+\bm 1} }
    } \frac{\abs{\dd \bm u}}{n_0^{\theta_1+\cdots+\theta_d}}
= \frac{\tau}{\,\,n_0^{\Theta_m}} \int_{\vv^o} \mh|{
    \frac{H_m(\bm u)}{
          \mb({ \bm 1-\nni \bm u }^{\bm n+\bm 1} }
    } \abs{\dd \bm u}
\end{equation}
for $n_0\ge n_0^*$, where $\vv^o$ is the image of $\cc^o$ under the change of variable, and the last equality used the \thetas[\theta_0\0m]-homogeneity of $H_m$.
The bounds \eqref{eq:main proof (1-u)^n bound}~and~\eqref{eq:main proof H(u) bound} imply that the integrand on the \rhs\ is bounded by
\begin{equation}
\frac{C |u_1\cdots u_d|^M}{c^d e^{-\mu \cdot \Re(u_1+\cdots +u_d)}}
\,=\, c^{-d}C \prod_{j=1}^d |u_j|^M e^{\mu \Re(u_j)} \,,
\end{equation}
which is integrable on $\vv \supset \vv^o$. Hence the inegral on $\vv^o$ is bounded by a constant independent of $n_0$. It follows that $\int_{\cc^o} \frac{o(H_m(\bm 1-\bm z))}{\bm z^{\bm n+\bm 1}} \dd \bm z = o(n^{-\Theta_m})$,
Adding this to \eqref{eq:m-transfer proof little-o localization} gives $I_m=o(n_0^{-\Theta_m})$.
Together with the conclusions of Steps~2~and~3, this implies the asymptotic expansion \eqref{eq:m-transfer coeff expansion} in Theorem~\ref{thm:m-transfer bis}.

When the little-o estimate in \eqref{eq:m-transfer fun expansion} is replaced by a big-O, the same proof (actually simpler, since one no longer needs to localize $\cc^O$ to $\cc^o$) shows the expansion \eqref{eq:m-transfer coeff expansion} with $o(n_0^{-\Theta_m})$ replaced by $O(n_0^{-\Theta_m})$.

\pdfstringdefDisableCommands{\renewcommand{\bm}[1]{#1}}
\vspace{-1em}
\paragraph{Step~4.~Coefficient asymptotics of homogeneous functions.}
It remains to prove the asymptotic expansion \eqref{eq:m-transfer power expansion} for a general \thetahomo\ function $H$ satisfying the assumptions of \Cref{thm:m-transfer}. With the same argument as in Step~3, there exists $r\in (0,1)$ such that
\begin{equation}\label{eq:main proof defining J}
\an[H(\bm 1-\bm z)] = J + O(r^{n_0^{\theta\1{min}}})
\qtq{with}
J := \pid \int_{\cc^O} \frac{H(\bm 1-\bm z)}{\bm z^{\bm n+\bm 1}}\dd \bm z
\end{equation}
We perform the same change of variable as in Step~3, which gives
\begin{equation}\label{eq:main proof H coeff}
J= \pid \int_{\vv^O} \frac{
    H(\nni \bm u)}{(1-\nni \bm u)^{\bm n + \bm 1}
    } \frac{\dd \bm u}{n_0^{\theta_1+\cdots + \theta_d}}
 = \frac1{n_0^\Theta} \cdot \pid \int_{\vv^O} \frac{
    H(\bm u)}{(1-\nni \bm u)^{\bm n + \bm 1}
    } \dd \bm u
\end{equation}
where $\Theta=\theta_0+\cdots+\theta_d$ as defined in \Cref{thm:m-transfer}.
According to the definition \eqref{eq:u-transfer g-factors} of the coefficients $g_{k,l}$, we have
\begin{equation}\label{eq:main proof G expansion}
\m({ 1-n_0^{-\theta}u }^{-\lambda n_0^\theta - 1}
  =\ e^{\lambda u}\, G\m({n_0^{-\theta}u, \lambda u}
\ =\ e^{\lambda u} \sum_{k=0}^\infty \m({ \sum_{l=0}^k g_{k,l}\,\lambda^l u^{k+l} } \frac1{n_0^{k\cdot \theta}} \,.
\end{equation}
Applying the above formula to each factor of the product
$\prod_{j=1}^d \mb({1-\nji u_j}^{-n_j-1}$ gives
\begin{equation}\label{eq:main proof G vectorial expansion}
\frac{ 1}{ \mn({ \bm 1-\nni\bm u }^{\bm n +\bm 1} }
=\ e^{\bm{\lambda \cdot u}} \sum_{\bm k\in \natural^d} \mH({
     \sum_{\bm l\le \bm k} g_{\bm k,\bm l}\,
        \bm{\lambda^l} \bm u^{\bm k+\bm l}
     } \frac1{n_0^{\bm{k \cdot \theta}} } \,.
\end{equation}
If we treat the \rhs\ as a formal sum over $\bm k$, and approximate the contour $\vv^O$ by its limit $\vv$, then \eqref{eq:main proof H coeff} would imply heuristically that:
\begin{align*}
J\,&\approx\,\frac{1}{n^\Theta} \sum_{\bm k\in \natural^d}
           \frac{1}{n_0^{\bm{k\cdot \theta}}} \cdot
           \pid \int_{\vv}
           \mH({ \sum_{\bm l\le \bm k} g_{\bm k,\bm l}\,
                      \bm{\lambda^l} \bm u^{\bm k+\bm l}
               }
           e^{\bm{\lambda \cdot u}} H(\bm u) \dd \bm u
\\&=\,     \frac{1}{n^\Theta} \sum_{\bm k\in \natural^d}
           \frac{1}{n_0^{\bm{k\cdot \theta}}} \cdot
           \sum_{\bm l\le \bm k} g_{\bm k,\bm l}\,
                \bm{\lambda^l}
                \bm{\partial_\lambda}^{\bm k+\bm l}
                \m({ \pid \int_{\vv}
                     e^{\bm{\lambda \cdot u}}
                     H(\bm u) \dd \bm u
                   }
\\&=\,\frac{1}{n^\Theta} \sum_{\bm k\in \natural^d}
     \frac{1}{n_0^{\bm{k\cdot \theta}}} \cdot
     D_{\bm k} I(\bm \lambda) \,.
\end{align*}
where the function $I(\bm \lambda)$ and the differential operator $D_{\bm k}$ are defined as in \Cref{thm:m-transfer}.

Now let us show that the last line is indeed an asymptotic expansion of $J$. In other words, for any~$N>0$, we have
\begin{equation}\label{eq:main proof truncated J asymptotics}
n_0^\Theta\cdot J = \,
    \sum_{\bm{k\cdot \theta}<N}
    \frac{D_{\bm k} I(\bm \lambda)
        }{n_0^{\bm{k\cdot \theta}}
        } + O\m({ \frac1{n_0^N} }
\end{equation}
when $n_0\to \infty$. For this, let us go back to \eqref{eq:main proof G expansion}. It can be seen as the Taylor series expansion of the function $f_y(x) = (1-x)^{-y/x-1}$ around $x=0$ evaluated at $x=n_0^{-\theta}u$ and $y=\lambda u$. The corresponding Taylor expansion with remainder term writes
\begin{equation}\label{eq:main proof Taylor with remainder}
f_y(x) = e^y \sum_{k=0}^m
         \m({ \sum_{l=0}^k g_{k,l}\,y^l } x^k
       + R_m(x,y)
\quad\text{with}\quad
       R_m(x,y) = \frac1{m!} \int_0^x (x-\xi)^m
                    f_y\0{m+1}(\xi) \dd \xi \,.
\end{equation}
It is not hard to show by induction that there are functions $\varphi_{k,m}$ continuous on the unit disk, such that
\begin{equation}
f_y\0m(x) = f_y(x) \cdot \sum_{k=0}^m \varphi_{k,m}(x) y^k \,,
\end{equation}
It follows that for each $m$, there exists a constant $C_m>0$ such that $\mb|{f_y\0m(x)} \le C_m(1+ |y|^m) \cdot \mb|{f_y(x)}$ for all $|x|\le 1/2$ and $y \in \complex$.  Plugging this into the definition of $R_m(x,y)$ gives that
\begin{equation}
\forall |x|\le 1/2 ,\, \forall y \in \complex,\quad
|R_m(x,y)| \le \frac{C_m}{(m+1)!} (1+|y|^m) \cdot |x|^{m+1}\sup_{\xi \in [0,x]}\mb|{f_y(\xi)}
\end{equation}
Consider the case where $(x,y)=(\nji u,\lambda u)$ with $n_0>0$, $\lambda \in [\lambda\1{min},\lambda\1{max}]$ and $u\in \vj$. In this case the condition $|x|\le 1/2$ is satisfied because $|u|\le \epsilon \nj$ on $\vj$ and $\epsilon=\epsilon(\delta)$ is assumed to be small enough.
Since $|u|\ge 1$ on $\vj$, the term $(1+|y|^m)$ is bounded by a constant times $|u|^m$ uniformly for $u\in \vj$.  Moreover, for all $\xi\in [0,x]$, there exists $\tilde n_0>0$ such that $\xi=\tilde n_0\raisebox{1pt}{$\!\!^{-\theta_j}$} u$. Thus we can apply the bound \eqref{eq:main proof (1-u)^n bound} to see that $|f_y(\xi)| = \mb|{(1-\tilde n_0\raisebox{1pt}{$\!\!^{-\theta_j}$} u)^{-\lambda u-1} } \le c^{-1} e^{\mu \cdot \Re(u)}$ for all $\xi\in [0,x]$. It follows that
\begin{equation}
\forall u \in \vj,\quad
\abs{ R_m(\nji u,\lambda u) }
\le \tilde C_m \frac{|u|^{2m+1} e^{\mu \cdot \Re(u)} }{n_0^{(m+1)\theta_j}}
\end{equation}
for some constant $\tilde C_m$ depending only on $\delta$, $\lambda\1{min}$, $\lambda\1{max}$ and $m$. Then, the Taylor expansion \eqref{eq:main proof Taylor with remainder} becomes
\begin{equation}
(1-\nji u)^{-\lambda \nj-1} = e^{\lambda u}
\sum_{k=0}^m \m({ \sum_{l=0}^k g_{k,l}\,\lambda^l u^{k+l} } \frac1{n_0^{k\cdot \theta_j}} + O\m({ \frac{|u|^{2m+1} e^{\mu\cdot \Re(u)}}{n_0^{(m+1)\theta_j}} }\,,
\end{equation}
where the big-O estimate is uniform \wrt\ $u \in \vj$ as $n_0\to \infty$. Now take $m>N/\theta\1{min}$ and replace each factor in $\prod_{j=1}^d \mb({1-\nji u_j}^{-n_j-1}$ by the \rhs\ of the above formula. After expanding the resulting product, we obtain a \emph{finite} sum. By collecting all the terms of order $O(1/n_0^N)$ together, we obtain an expansion of the form
\begin{equation}
\frac{1}{ \mn({ \bm 1-\nni\bm u }^{\bm n +\bm 1} }
=\ e^{\bm{\lambda \cdot u}} \sum_{\bm{k\cdot \theta}<N} \mH({
     \sum_{\bm l\le \bm k} g_{\bm k,\bm l}\,
        \bm{\lambda^l} \bm u^{\bm k+\bm l}
     } \frac1{n_0^{\bm{k \cdot \theta}} }
   + O\m({ \frac{R(\bm u)}{n_0^N}  } ,
\end{equation}
where the big-O estimate is uniform on $\bm u\in \vv^O$. By following the above calculation more closely, it is not hard to see that we can choose $R(\bm u) = |u_1\cdots u_d|^{\tilde M} e^{\mu \cdot \Re(u_1+\cdots + u_d)}$ for some $\mu,\tilde M>0$.
With the bound \eqref{eq:main proof H(u) bound} for $H(\bm u)$, it follows that $H(\bm u)R(\bm u)$ is integrable on $\vv \supset \vv^O$. So we can integrable the previous display term by term \wrt\ $\bm u\in \vv^O$ to obtain
\begin{equation}
\int_{\vv^O} \frac{H(\bm u)}{
  \mn({ \bm 1-\nni\bm u }^{\bm n +\bm 1} }\dd \bm u
= \sum_{\bm{k\cdot \theta}<N}
   \int_{\vv^O} H(\bm u) e^{\bm{\lambda \cdot u}} \mH({
     \sum_{\bm l\le \bm k} g_{\bm k,\bm l}\,
        \bm{\lambda^l} \bm u^{\bm k+\bm l}
     }\dd \bm u
   \cdot \frac1{n_0^{\bm{k \cdot \theta}} }
   + O\m({ \frac1{n_0^N}  } .
\end{equation}
Like in Step~1, replacing the contour $\vv^O$ by $\vv$ in each term of the above equation only produces an error of order $O(r^{n_0^{\theta\1{min}}})$ as $n_0\to \infty$. The resulting equation divided by $(2\pi i)^d$ gives exactly \eqref{eq:main proof truncated J asymptotics}.

The expansion \eqref{eq:m-transfer power expansion} follows readily from \eqref{eq:main proof defining J} and \eqref{eq:main proof truncated J asymptotics}. This concludes the proof of Theorem~\ref{thm:m-transfer bis}. \qed

\section{The Borel-Laplace transforms}\label{sec:Borel-Laplace}

\newcommand*{\udom}[1][r]{U_{\delta'\!,#1}}
\newcommand*{\vd}[1][']{\bm V_{\delta#1}}

In this section, we prove the three statements in \Cref{thm:Borel-Laplace}.

\paragraph*{Proof of \Crefp{1}{thm:Borel-Laplace}.}

Fix $\delta\in (0,\frac\pi2)$ and $H\in \poly(\odom)$. 
Recall that $\borel[H]$ is defined as the integral 
\begin{equation}
I_{\vd}(\bm \lambda) 
:= \pid \int_{\vv_{\delta'}} 
e^{\bm \lambda \cdot \bm u} H(\bm u) \dd \bm u,
\end{equation}
on a contour $\vd$ chosen from the class of contours
\begin{align*}
\bm{\mathcal V}_{\delta'} := 
\big\{ V_1 \times \cdots \times V_d \,\big|\, & \,
\forall 1\le j\le d,\, V_j \text{ is a piecewise smooth curve in } \Omega_\delta \text{ } \\
&\ \text{which coincide with }\partial\Omega_{\delta'}
   \text{ outside a bounded set, as in \Crefp{a}{fig:Borel-Laplace proof}}
\big\} \,.
\end{align*}
Thanks to the analyticity of $H$ on the domain $\odom=(\Omega_{\delta})^d$, it is clear that for each fixed $\delta'$, the value of the integral $I_{\vd}(\bm \lambda)$ does not depend on the choice of the contour $\vd$ within the class $\bm{\mathcal V}_{\delta'}$.

First, let us fix $\delta' \in (0,\delta)$ and $\vd \in \bm{\mathcal V}_{\delta'}$ and show that the integral $I_{\vd}(\bm \lambda)$ is absolutely convergent and analytic \wrt\ $\bm \lambda \in \kdom[\delta']$.
Let $\udom$ be the closure of $(\complex \setminus \Omega_{\delta'}) \cup B_{0,r}$, where $B_{0,r}\subset \complex$ is the disk of radius $r$ around the origin.
It is a simple exercise to show that for $\delta^\circ \in (0,\delta')$ and $r>0$, we have
\begin{equation}\label{eq:Borel-Laplace proof Borel basic bound}
\forall u \in \udom,\ 
\forall \lambda \in K_{\delta^\circ},\quad
\mb|{ e^{\lambda u} } 
\, \le \,
C_*^{r|\lambda|} e^{-\sigma_* |\lambda|\cdot |u|} \,,
\end{equation}
with $\sigma_* = \sin(\delta'-\delta^\circ)>0$ and $ C_*=e^{1+\sigma_*}$.
Since each component of the contour $\vd \in \bm{\mathcal V}_{\delta'}$ is bounded away from the origin and $H$ is of polynomial type, there exist $\tilde C,M>0$ such that
\begin{equation}\label{eq:Borel-Laplace proof Borel poly bound}
\forall \bm u\in \vv_{\delta'},\quad
|H(\bm u)| \le \tilde C \cdot \m({ |u_1|^M + \cdots + |u_d|^M }  \,.
\end{equation}
On the other hand, there exists $r>0$ such that $V_{\delta'} \subset \udom$. So the bound \eqref{eq:Borel-Laplace proof Borel basic bound} implies that for any bounded set $\bm S\subseteq \kdom[\delta^\circ]$, there exists $C,\sigma>0$, such that
\begin{equation}\label{eq:Borel-Laplace proof Borel exp bound}
\forall \bm u \in \vv_{\delta'} ,\ 
\forall \bm \lambda \in \bm S,\quad
\mb|{ e^{\bm \lambda \cdot \bm u} }
\le C \cdot e^{-\sigma (|u_1|+\cdots + |u_d|)} \,.
\end{equation}
Up to increasing the value of $C$ and decreasing $\sigma$, the polynomial bound on the \rhs\ of \eqref{eq:Borel-Laplace proof Borel poly bound} can be absorbed by the exponential decay in the above display. Therefore we have
\begin{equation}\label{eq:Borel-Laplace proof Borel integrand bound}
\forall \bm u \in \vv_{\delta'} ,\ 
\forall \bm \lambda \in \bm S,\quad
\mb|{ e^{\bm \lambda \cdot \bm u} H(\bm u) }
\le C \cdot e^{-\sigma (|u_1|+\cdots + |u_d|)} 
\end{equation}
for some $C,\sigma>0$.
Since the \rhs\ is independent of $\bm \lambda\in \bm S$ and integrable on $\vv_{\delta'}$, it follows that the integral $I_{\vd}(\bm \lambda)$ is absolutely convergent and analytic \wrt\ $\bm \lambda \in \bm S$. And since this is true for all $\delta^\circ \in (0,\delta')$ and bounded set $\bm S \subset \kdom[\delta^\circ]$, the integral $I_{\vd}$ defines an analytic function on $\kdom[\delta']$. 

Now let us show that $I_{\vd}(\bm \lambda)$ is independent of $\delta'\in(0,\delta)$ as well. For this, we fix $\delta',\delta''\in (0,\delta)$ and $\bm \lambda \in \kdom[\min(\delta',\delta'')]$.
For each $N>0$, let $V^N_{\delta''}$ be the contour obtained by deforming each component of $\vd['']$ inside the disk $B_{0,N}$ to coincide with the corresponding component of $\vd$ there, while keeping it unchanged outside $B_{0,N}$. See \Crefp{a}{fig:Borel-Laplace proof}.
By the triangular inequality, we have
\begin{equation}\label{eq:Borel-Laplace proof symdiff bound}
\abs{ 
\m({\int_{\vv_{\delta' }} - \int_{\vv_{\delta''}} }
e^{\bm \lambda \cdot \bm u} H(\bm u) \dd \bm u
}
=
\abs{ 
\m({\int_{\vv_{\delta' }} - \int_{\vv^N_{\delta''}} } 
e^{\bm \lambda \cdot \bm u} H(\bm u) \dd \bm u
}
\le 
\int_{ \vv_{\delta' } \,\Delta\, \vv^N_{\delta''} }
\abs{ e^{\bm \lambda \cdot \bm u} H(\bm u)} \cdot \abs{\dd \bm u} \,,
\end{equation}
where $\vv_{\delta' }\, \Delta\, \vv^N_{\delta''}$ is the symmetric difference between $\vv_{\delta' }$ and $\vv^N_{\delta''}$.
Let
\begin{equation}
\bm S_n = \setb{\bm u \in \vv_{\delta' } \,\Delta\, \vv^N_{\delta''} }{ n\le \max_j|u_j| < n+1} \,.
\end{equation}
It is not hard to see that $\bm S_n=\varnothing$ for $n<N$, and the $d$-dimensional Lebesgue measure of $\bm S_n$ is bounded by $c n^d$ for some constant $c=c(\delta',\delta'',d)$. 
Moreover, since each point $\bm \lambda \in \kdom[\min(\delta',\delta'')]$ belongs to $\kdom[\delta^\circ]$ for some $\delta^\circ < \min(\delta',\delta'')$, we can deduce from \eqref{eq:Borel-Laplace proof Borel integrand bound} that $\abs{ e^{\bm \lambda \cdot \bm u} H(\bm u)} \le C e^{-\sigma n}$ for all $\bm u\in \bm S_n$ and all $n$. (Here we also use the observation that \eqref{eq:Borel-Laplace proof Borel integrand bound} remains valid when $\vd$ is replaced by $\vd^N$). It follows that
\begin{equation}
\int_{ \vv_{\delta' } \Delta \vv^N_{\delta''} }
\abs{ e^{\bm \lambda \cdot \bm u} H(\bm u)} \cdot \abs{\dd \bm u}
\ \le\ 
\sum_{n=N}^\infty C e^{-\sigma n} \cdot cn^d
\ \cv[]N\ 0 \,.
\end{equation}
Together with \eqref{eq:Borel-Laplace proof symdiff bound}, this implies that $I_{\vd}(\bm \lambda) = I_{\vd['']}(\bm \lambda)$ for all $\bm \lambda \in \kdom[\min(\delta',\delta'')]$.
In this sense, $\borel[H]=I_{\vd}$ is independent of $\delta'\in (0,\delta)$ and thus defines an analytic function on $\kdom$.

\begin{figure}
\centering
\includegraphics[scale=1,page=2]{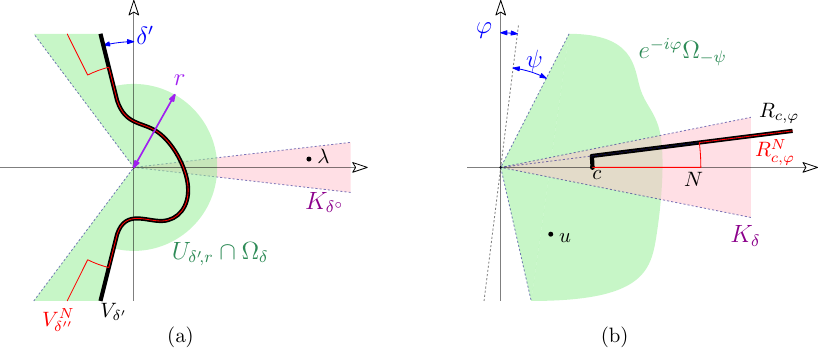}
\caption{(a) Some domains and contours used in the proof of \Crefp{1}{thm:Borel-Laplace}.
(b) Some domains and contours used in the proof of \Crefp{2}{thm:Borel-Laplace}.}
\label{fig:Borel-Laplace proof}
\end{figure}

To show that $\borel[H] \in \poly(\kdom[\delta^\circ])$ for all $\delta^\circ\in (0,\delta)$, let us fix $\bm \lambda \in \kdom[\delta^\circ]$ and $\delta'\in (\delta^\circ,\delta)$, and consider the particular contour $\vv_{\delta'} = V_1 \times \cdots V_d$ with $V_j = \partial \udom[1/|\lambda_j|]$, where $\udom$ is defined above \eqref{eq:Borel-Laplace proof Borel basic bound}.
By the bound \eqref{eq:Borel-Laplace proof Borel basic bound}, we have 
\begin{equation}
\forall \bm u \in \vv_{\delta'},\quad
\mb|{e^{\bm \lambda \cdot \bm u}} 
\le  C_*^d \cdot e^{-\sigma_* (|\lambda_1 u_1|+ \cdots + |\lambda_d u_d|)}
\end{equation}
Since the contour $V_j$ is not bounded away from the origin when $|\lambda_j|\to \infty$, we need to use the complete bound \eqref{eq:def polynomial type 2-ended} for the function of polynomial type $H\in \poly(\odom)$ here, which implies 
\begin{equation}
\forall \bm u \in \vv_{\delta'},\quad
\mb|{ e^{\bm \lambda \cdot \bm u} H(\bm u) }
\le  C \cdot \mB({ |u_1|^M + |u_1|^{-M} + \cdots + |u_d|^M + |u_d|^{-M} } \cdot e^{-\sigma_* \cdot (|\lambda_1 u_1|+ \cdots + |\lambda_d u_d|)} 
\end{equation}
for some $C,M>0$. We integrate the above bound on $\vv_{\delta'}$ and make the change of variables $v_j= \lambda_j u_j$. Notice that the contour $\bm \lambda \vd \equiv (\lambda_1 V_1) \times \cdots \times (\lambda_d V_d)$ for the variable $\bm v$ after this change no longer depends on $\bm \lambda$. This gives us
\begin{align*}
\mb|{ \borel[H](\bm \lambda) } 
&= \abs{\pid \int_{\vv_{\delta'}} 
        e^{\bm \lambda \cdot \bm u} H(\bm u) \dd \bm u 
   } \\
&\le 
\frac C{(2\pi)^d} \int_{\bm \lambda \vv_{\delta'}}
\m({ \abs{\frac{v_1}{\lambda_1}}^M \!\!
   + \abs{\frac{v_1}{\lambda_1}}^{-M} \!\! + \cdots
   + \abs{\frac{v_d}{\lambda_d}}^M \!\!
   + \abs{\frac{v_d}{\lambda_d}}^{-M}
    } 
\frac{  e^{-\sigma_* (|v_1| + \cdots + |v_d|)}
|\dd \bm v|}{|\lambda_1 \cdots \lambda_d|} \\
&\le \frac{\max \m({ |\lambda_1|^M, |\lambda_1|^{-M},
           \ldots, |\lambda_d|^M, |\lambda_d|^{-M}
     }}{|\lambda_1 \cdots \lambda_d|}
     \cdot J
\end{align*}
where the integral
\begin{equation}
J = 
\frac C{(2\pi)^d} \int_{\bm \lambda \vv_{\delta'}}
\m({  |v_1|^M + |v_1|^{-M} + \cdots 
    + |v_d|^M + |v_d|^{-M} } 
\cdot e^{-\sigma_* (|v_1| + \cdots + |v_d|)}
\abs{ \dd \bm v }
\end{equation}
is finite and independent of $\bm \lambda$. 
It follows that
\begin{equation}
\forall \bm \lambda \in \kdom[\delta^\circ],\quad
\mb|{ \borel[H](\bm \lambda) }
\le J\cdot \frac{|\lambda_1|^M + |\lambda_1|^{-M}
      + \cdots + |\lambda_d|^M + |\lambda_d|^{-M}
    }{|\lambda_1 \cdots \lambda_d|} \,.
\end{equation}
We will see in \Cref{lem:poly-type closure} that the monomial function $\bm \lambda \mapsto \bm{\lambda^\alpha}$ is in $\poly(\kdom[\delta^\circ])$, for all $\bm \alpha \in \real^d$ and $\delta^\circ\in (0,\pi)$.
This allows us to bound the \rhs\ of the above display by a polynomial-type bound of the form \eqref{eq:def polynomial type 2-ended}.
Hence we have $\borel[H]\in \poly(\kdom[\delta^\circ])$, for all $\delta^\circ\in (0,\delta)$.

\paragraph*{Proof of \Crefp{2}{thm:Borel-Laplace}.}
Fix $\bm c \in \real_{>0}^d$, $\delta \in (0,\frac\pi2)$ and $I\in \poly(\kdom)$. 
For each $\bm \varphi \in (-\delta,\delta)^d$, we define $\bm{R_\varphi} = R_{c_1,\varphi_1} \times \cdots \times R_{c_d,\varphi_d} \subset \kdom$, where $R_{c,\varphi} \subset \complex$ is the contour which starts at $c\in \real_{>0}$, then follows an arc on the circle $\partial B_{0,|c|}$, and then goes to $\infty$ along the ray $\set{\lambda\in \complex_*}{\arg(\lambda)=\varphi}$, as in \Crefp{b}{fig:Borel-Laplace proof}.
Let
\begin{equation}
H_{\bm \varphi}(\bm u) = \int_{\bm{R_\varphi}} e^{-\bm \lambda \cdot \bm u} I(\bm \lambda) \dd \bm \lambda 
\,.
\end{equation}
Recall that $\odom[\psi] = (\Omega_\psi)^d$ with $\Omega_\psi \equiv K_{\frac\pi2+\psi} := \Set{u\in \complex_*}{ |\arg(u)|<\frac\pi2+\psi }$, for any $\psi\in (-\frac\pi2,\frac\pi2)$.

First, let us show that the integral $H_{\bm \varphi}(\bm u)$ is absolutely convergent and analytic \wrt\ $\bm u \in e^{-i\bm \varphi} \odom[0] := (e^{-i\varphi_1} \Omega_0) \times \cdots \times (e^{-i\varphi_d} \Omega_0)$. The proof is similar to the one for the integral $I_{\vd}(\bm \lambda)$: 
It~is a simple exercise to show that for any $\varphi\in (-\delta,\delta)$, $\psi\in (0,\frac\pi2)$ and $c\in \{c_1,\ldots,c_d\}$ we have
\begin{equation}\label{eq:Borel-Laplace proof Laplace basic bound}
\forall \lambda \in R_{c,\varphi} ,\
\forall u \in e^{-i\varphi} \Omega_{-\psi},\quad
\mb|{ e^{-\lambda u} } \le
C_*^{|u|} e^{-\sigma_* |\lambda| \cdot |u|}
\end{equation}
with $\sigma_*=\sin(\psi)>0$ and $C_*=e^{(1+\sigma_*)\max_j |c_j|}$. Following the same steps as for the bounds \eqref{eq:Borel-Laplace proof Borel poly bound}--\eqref{eq:Borel-Laplace proof Borel integrand bound}, we deduce that for any bounded subset $\bm S\subset e^{-i\bm \varphi}\odom[-\psi]$, where $\psi\in (0,\frac\pi2)$ and $\bm \varphi \in (-\delta,\delta)^d$, there exists $C,\sigma>0$ such that
\begin{equation}\label{eq:Borel-Laplace proof Laplace integrand bound}
\forall \bm \lambda\in \bm{R_\varphi} ,\
\forall \bm u \in \bm S,\quad
\mb|{ e^{\bm \lambda \cdot \bm u} I(\bm \lambda) }
\le C\cdot e^{-\sigma( |\lambda_1|+\cdots+|\lambda_d| )} \,.
\end{equation}
Since the \rhs\ is independent of $\bm u\in \bm S$ and integrable on $\bm{R_\varphi}$, it follows that the integral $H_{\bm \varphi}(\bm u)$ is absolutely convergent and analytic \wrt\ $\bm u \in \bm S$. And since this is true for all bounded $\bm S\subset e^{-i\bm \varphi} \odom[-\psi]$ with $\psi\in (0,\frac\pi2)$, we conclude that $H_{\bm \varphi}$ defines an analytic function on $e^{-i\bm \varphi} \odom[0]$.

Now let us show that $H_{\bm \varphi}(\bm u) = H_{\bm 0}(\bm u)$ for all $\bm \varphi \in (-\delta,\delta)^d$ and $\bm u \in \odom[-\delta]$. 
Fix some $\bm \varphi \in (-\delta,\delta)^d$. For any $N>0$, let $\bm{R_\varphi}^N = R_{c_1,\varphi_1}^N \times \cdots \times R_{c_d,\varphi_d}^N$, where $R_{c,\varphi}^N$ is the curve which coincides with the interval $[c,N]$ inside the disk $B_{0,N}$, while staying on the ray $\set{\lambda\in \complex_*}{\arg(\lambda)=\varphi}$ outside $B_{0,N}$. See \Crefp{b}{fig:Borel-Laplace proof}. 
By construction, the contour $\bm{R_\varphi}^N$ coincides with $\bm{R_0}$ in the polydisk $(B_{0,N})^d$. Like in the previous proof for $I_{\vd}(\bm \lambda)$, the sets
\begin{equation}
\tilde{\bm S}_n = 
\setb{\bm \lambda \in \bm{R_\varphi}^N \,\Delta\, \bm{R_0}
}{ n\le \max_j |\lambda_j|<n+1 }
\end{equation}
satisfy $\tilde{\bm S}_n = \varnothing$ for all $n<N$ and $\int_{\tilde{\bm S}_n}|\dd \bm \lambda| \le cn^d$ for some $c<\infty$ independent of $n$.
Moreover, since each point $\bm u\in \odom[-\delta]$ belongs to $\odom[-\psi] \cap e^{-i\bm \varphi}\odom[-\psi]$ for some $\psi\in (0,\frac\pi2)$, we can deduce from \eqref{eq:Borel-Laplace proof Laplace integrand bound} that $\mb|{ e^{-\bm \lambda \cdot \bm u} I(\bm \lambda) } \le C e^{-\sigma n}$ for all $\bm \lambda \in \tilde{\bm S}_n$ and all $n$. (Here we also use the fact that \eqref{eq:Borel-Laplace proof Laplace integrand bound} remains valid when $\bm{R_\varphi}$ is replaced by $\bm{R_\varphi}^N$).
It follows that
\begin{align*}
\mB|{ H_{\bm \varphi}(\bm u) - H_{\bm 0}(\bm u) }
&= \abs{ \m({ \int_{\bm{R_\varphi}^N} 
            - \int_{\bm{R_0}} }
            e^{-\bm \lambda\cdot \bm u} I(\bm \lambda)
            \dd \bm \lambda }
\le\int_{ \bm{R_\varphi}^N \Delta \bm{R_0} }
    \abs{ e^{-\bm \lambda\cdot \bm u} I(\bm \lambda) }
    \abs{\dd \bm \lambda} \\
&\le \sum_{n=N}^\infty Ce^{-\sigma n} \cdot cn^d
\ \cv[]N \ 0
\end{align*}
Therefore, $H_{\bm \varphi}(\bm u) = H_{\bm 0}(\bm u)$ for all $\bm \varphi \in (-\delta,\delta)^d$ and $\bm u \in \odom[-\delta]$. And since each $H_{\bm \varphi}$ is analytic on $e^{-i\bm \varphi} \odom[0]$, the function $\laplace{}[I]\equiv H_{\bm 0}$ has an analytic continuation on the union $\odom = \bigcup_{\bm \varphi \in (-\delta,\delta)^d} e^{-i\bm \varphi} \odom[0]$.

\newcommand{\bhat}{\tilde{\bm B}}

Finally, let us fix $\delta^\circ \in (0,\delta)$ and $\bm \lambda \in \kdom[\delta^\circ]$, and show that $\borel \circ \laplace{}[I]$ is well-defined and analytic at $\bm \lambda$ when the $|\lambda_j|$'s are large enough. 
For this, let us first prove that for all $\delta''\in (0,\delta)$, there exists $m>0$ such that
\begin{equation}\label{eq:Borel-Laplace proof Laplace value exp bound}
\forall \bm u \in \odom[\delta''] \cap \bhat,\quad
\abs{ \laplace{}[I](\bm u) } 
\le e^{m\cdot (|u_1| + \cdots + |u_d|)} \,.
\end{equation}
where $\bhat := \setn{\bm u\in \complex^d}{\forall j,\,|u_j|\ge 1}$.
Indeed, for all $\bm \varphi \in (-\delta,\delta)^d$ and $\psi\in (0,\frac\pi2)$, \eqref{eq:Borel-Laplace proof Laplace basic bound} implies that
\begin{equation}\label{eq:Borel-Laplace proof Laplace exp bound}
\forall \bm u \in e^{-i\bm \varphi} \odom[-\psi],\ 
\forall \bm \lambda \in \bm{R_\varphi} ,\quad
\mb|{ e^{-\bm \lambda \cdot \bm u} }
 \le C_*^{|u_1|+ \cdots + |u_d|}
\cdot e^{-\sigma_* \m({ |\lambda_1 u_1|+ \cdots + |\lambda_d u_d| }} \,.
\end{equation}
In addition, since $I\in \poly(\kdom)$ and $\bm{R_\varphi} \subset \kdom$ is bounded away from $\bm 0$, there exist $\tilde C,M>0$ such that
\begin{equation}\label{eq:Borel-Laplace proof Laplace poly bound}
\forall \bm \lambda \in \bm{R_\varphi} ,\quad
\mb|{ I(\bm \lambda) } 
 \le \tilde C \m({ |\lambda_1|^M + \cdots + |\lambda_d|^M } \,.
\end{equation}
For $\bm u\in \bhat$, we can increase the value of $C_*$ and decrease that of $\sigma_*$ on the \rhs\ of \eqref{eq:Borel-Laplace proof Laplace exp bound} to absorbe the polynomial function on \rhs\ of \eqref{eq:Borel-Laplace proof Laplace poly bound}. It follows that there exist constants~$C,\sigma>0$, which do not depend on $\bm \varphi$, such that
\begin{equation}
\forall \bm u \in e^{-i\bm \varphi} \odom[-\psi]\! \cap \! \bhat ,\ 
\forall \bm \lambda \in \!\bm{R_\varphi} ,\quad
\mb|{ e^{-\bm \lambda \cdot \bm u}I(\bm \lambda) }
\le C^{|u_1|+ \cdots + |u_d|}
\cdot e^{-\sigma \m({ |\lambda_1 u_1|+ \cdots + |\lambda_d u_d| }}.
\end{equation}
One can easily check by direct computation that
\begin{equation}
J := \sup_{\bm \varphi \in (-\delta,\delta)^d} \  \sup_{\bm u\in \bhat} \
\int_{\bm{R_\varphi}} \!
e^{-\sigma \m({ |\lambda_1 u_1|+ \cdots + |\lambda_d u_d| }} |\dd \bm \lambda|
\end{equation}
is finite. It follows that $|\laplace{}[I](\bm u)| \le J\cdot C^{|u_1| + \cdots + |u_d|}$ for all $\bm u\in e^{-i\bm \varphi} \odom[-\psi] \cap \bhat$. This implies \eqref{eq:Borel-Laplace proof Laplace value exp bound},
because $J,C$ are independent of $\bm \varphi$, and we have $\odom[\delta''] = \bigcup_{\bm \varphi \in (-\delta,\delta)^d} e^{-i\bm \varphi} \odom[-\psi]$ with the choice $\psi:=\delta-\delta''$. 
Now assume $\delta^\circ < \delta'<\delta''$ and consider the contour $\vd = (\partial \udom[1])^d$, where $\udom$ is defined above \eqref{eq:Borel-Laplace proof Borel basic bound} (see also \Crefp{a}{fig:Borel-Laplace proof}). Then \eqref{eq:Borel-Laplace proof Borel basic bound} implies that
\begin{equation}
\forall \bm \lambda \in \kdom[\delta^\circ] ,\ 
\forall \bm u \in \vd ,\quad
\mb|{ e^{\bm \lambda \cdot \bm u} }
\le C_*^{|\lambda_1|+\cdots+|\lambda_d|} 
e^{-\sigma_* \m({ |\lambda_1 u_1| + \cdots + |\lambda_d u_d| }} 
\end{equation}
Since we have $\vd \subset \odom[\delta''] \cap \bhat$, the last display and \eqref{eq:Borel-Laplace proof Laplace value exp bound} imply that $\borel[\laplace{}[I]]$ is well-defined and analytic on $\Setb{\bm \lambda \in \kdom[\delta^\circ]}{\forall j,\, |\lambda_j|>\frac{m+1}{\sigma_*}}$.

\paragraph*{Proof of \Crefp{3}{thm:Borel-Laplace}.}

Notice that the Borel transform $\borel$ and the Laplace transform $\laplace$ can both be written as the product of $d$ univariate integral transforms: we have $\borel = \borel\01 \circ \cdots \circ \borel\0d$ and $\laplace = \laplace[c_1]\01 \circ \cdots \circ \laplace[c_d]\01$ with
\begin{equation}
\borel\0j[H](\bm u_{\hat j}) := \frac{1}{2\pi i} \int_{V_{\delta'}} e^{\lambda u_j} H(\bm u) \dd u_j
\qtq{and}
\laplace[c]\0j[I](\bm \lambda_{\hat j}) := \int_c^\infty e^{\lambda_j u} I(\bm \lambda) \dd \lambda_j,
\end{equation}
where $\bm u_{\hat j} = (u_1,\ldots,u_{j-1},\lambda,u_{j+1},\ldots u_d)$ and $\bm \lambda_{\hat j} = (\lambda_1,\ldots,\lambda_{j-1},u,\lambda_{j+1},\ldots \lambda_d)$. Moreover, for all~$j\ne k$, $\borel\0j$ commute with $\laplace[c_k]\0k$, because they operate on different variables. It follows that 
\begin{equation}\label{eq:Borel-Laplace proof univariate decomposition}
\borel \circ \laplace{} = \m({\borel\01 \circ \laplace[c_1]\01} \circ \cdots \circ \m({\borel\0d \circ \laplace[c_d]\0d}
\qtq{and}
\laplace{} \circ \borel = \m({\laplace[c_1]\01 \circ \borel\01} \circ \cdots \circ \m({\laplace[c_d]\0d \circ \borel\0d} .
\end{equation}
The above decompositions allow us to reduce the proof of \Crefp{3}{thm:Borel-Laplace} to the univariate case. In the following, we assume $d=1$ and fix some $\delta\in (0,\frac\pi2)$ and $c\in \real_{>0}$. 

Let $I\in \poly(K_\delta)$. Let us show that $\borel \circ \laplace[c][I](\lambda) = I(\lambda)$ for all $\lambda>c$.
For this, fix some $\delta'\in (0,\delta)$ and let  $V_{\delta'} = \partial \odom$ with the domain $\odom$ defined above \eqref{eq:Borel-Laplace proof Borel basic bound}. 
Let $V_{\delta'}^+$ (resp.\ $V_{\delta'}^-$) be the part of $V_{\delta'}$ in the upper half-plane (resp.\ lower half-plane). 
Recall from the proof of \Crefp{2}{thm:Borel-Laplace} that $\laplace[c][I]$ can be expressed as an integral over $R_{c,\varphi}$ for any $\varphi \in (-\delta,\delta)$, where $R_{c,\varphi}$ is the contour shown in \Crefp{3}{fig:Borel-Laplace proof}. Fix some $\varphi \in (\delta',\delta)$. Then we can write
\begin{align*}
\borel \circ \laplace[c][I](\lambda) 
&= \frac{1}{2\pi i} \int_{V_{\delta'}} 
   e^{\lambda u} \laplace[c][I](u) \dd u \\
&= \frac{1}{2\pi i} \m({ 
   \int_{V_{\delta'}^+} e^{\lambda u} 
   \m({ \int_{R_{c,-\varphi}} e^{-\tau u} I(\tau) \dd \tau }
   \dd u
 + \int_{V_{\delta'}^-} e^{\lambda u} 
   \m({ \int_{R_{c, \varphi}} e^{-\tau u} I(\tau) \dd \tau }
   \dd u
}.
\end{align*}
Thanks to the bounds \eqref{eq:Borel-Laplace proof Borel basic bound} and \eqref{eq:Borel-Laplace proof Laplace basic bound}, it is not hard to see that $(u,\tau)\mapsto e^{\lambda u} e^{-\tau u} I(\tau)$ is integrable on both $V_{\delta'}^+ \times R_{c,-\varphi}$ and $V_{\delta'}^- \times R_{c,\varphi}$. Thus we have by Fubini's theorem
\begin{align*}
\borel \circ \laplace[c][I](\lambda) 
&= \frac1{2\pi i} \m({ 
     \int_{R_{c,-\varphi}} I(\tau) 
     \m({ \int_{V_{\delta'}^+} e^{\lambda u-\tau u} \dd u }
     \dd \tau
   + \int_{R_{c, \varphi}} I(\tau) 
     \m({ \int_{V_{\delta'}^-} e^{\lambda u-\tau u} \dd u }
     \dd \tau
} \\
&= \frac1{2\pi i} \m({ 
     \int_{R_{c,-\varphi}} I(\tau) \cdot
     \m[{ \frac{ e^{(\lambda-\tau) u} }{\lambda-\tau} 
        }_1^{e^{i(\frac\pi2+\delta')}\infty}
     \!\!\! \dd \tau
\ +\ \int_{R_{c, \varphi}} I(\tau) \cdot
     \m[{ \frac{ e^{(\lambda-\tau) u} }{\lambda-\tau} 
        }^1_{e^{-i(\frac\pi2+\delta')}\infty}
     \!\!\! \dd \tau
} \\
&= \frac1{2\pi i} \m({ 
     \int_{R_{c,-\varphi}} I(\tau) \cdot
     \frac{ e^{\lambda-\tau} }{\tau-\lambda} 
     \dd \tau
   + \int_{R_{c, \varphi}} I(\tau) \cdot
     \frac{ e^{\lambda-\tau} }{\lambda-\tau} 
     \dd \tau
} \\
&= \frac1{2\pi i} 
     \int_{\mathcal R} 
     I(\tau) \cdot
     \frac{ e^{\lambda-\tau} }{\tau-\lambda} 
     \dd \tau
\end{align*}
where $\mathcal R = R_{c,-\varphi} \cup \tilde R_{c,\varphi}$ and $\tilde R_{c,\varphi}$ is the contour $R_{c,\varphi}$ oriented in the opposite direction. Notice that $\mathcal R \subset K_\delta$ is a bi-infinite contour whose two ends extend to $\infty$. Moreover, the integrand $\tau \mapsto I(\tau) \frac{ e^{\lambda-\tau} }{\tau-\lambda}$ is meromorphic on $K_\delta$, has a unique simple pole at $\tau=\lambda$ (which is on the right of the contour $\mathcal R$), and decays exponentially when $\Re(\tau)\to \infty$. 
It follows from the residue theorem that
\begin{equation}
\frac1{2\pi i} 
     \int_{\mathcal R} 
     I(\tau) \cdot
     \frac{ e^{\lambda-\tau} }{\tau-\lambda} 
     \dd \tau
= \res_{\tau \to \lambda} \m({ I(\tau)
     \frac{ e^{\lambda-\tau} }{\tau-\lambda} }
= I(\lambda) \,.
\end{equation}
This proves $\borel \circ \laplace[c][I](\lambda) =I(\lambda)$ for all $\lambda>c$, and thus for all $\lambda \in K_\delta$ by analytic continuation. In other words, $\borel \circ \laplace[c]$ is the identity on $\poly(K_{\delta})$. By the decomposition \eqref{eq:Borel-Laplace proof univariate decomposition}, the same is true in any dimension $d$.

Now fix some $H\in \poly(\Omega_\delta)$ and $u>0$. Consider the  contour $V_{\delta'} = \partial \udom$ for some $\delta'\in (0,\delta)$ and $r\in (0,u)$. It is not hard to check that the function $(\lambda,v)\mapsto e^{-\lambda u} e^{\lambda v} H(v)$ is integrable on $[c,\infty)\times V_{\delta'}$. By Fubini's theorem, we have
\begin{align}
\laplace[c]\circ \borel[H] (u)
&= \frac{1}{2\pi i} \int_c^\infty e^{-\lambda u} 
   \m({ \int_{V_{\delta'}} e^{\lambda v} H(v) \dd v)
      } \dd \lambda \\
&= \frac{1}{2\pi i} \int_{V_{\delta'}} H(v) 
   \m({ \int_c^\infty e^{-(u-v)\lambda} \dd \lambda 
      } \dd v 
= \frac{1}{2\pi i} \int_{V_{\delta'}} H(v) 
   \frac{ e^{-(u-v)c} }{u-v} \dd v \,.
\end{align}
For each $u\in \complex$, let
\begin{equation}
E_c(u) = - \frac{1}{2\pi i} \lim_{R\to \infty} \int_{\partial \udom[R]} H(v) \frac{e^{-(u-v)c}}{u-v} \dd v \,.
\end{equation}
It is not hard to see that the above limit stablizes when $R>|u|$, and defines an entire function of $u$. Actually, by the residue theorem, we have for all $R>|u|$,
\begin{align*}
\frac{1}{2\pi i} 
\int_{\partial \udom[R]} H(v) \frac{e^{-(u-v)c}}{u-v} \dd v 
&= \frac{1}{2\pi i} 
\int_{V_{\delta'}} H(v) \frac{ e^{-(u-v)c} }{u-v} \dd v
+ \res_{v\to u} H(v) \frac{ e^{-(u-v)c} }{u-v}  \\
&= \laplace[c]\circ \borel[H] (u) - H(u) \,.
\end{align*}
It follows that $\laplace[c]\circ \borel[H] (u) = H(u) + E_c(u)$ for all $u>0$. By analytic continuation, the same is true for all $u\in \Omega_\delta$.
To recover the case of general dimension $d$, we apply this formula of $\laplace[c]\circ \borel[H]$ to each factor in the decomposition \eqref{eq:Borel-Laplace proof univariate decomposition}. To simplify notation, we write $\mathcal I_j =
\laplace[c_j]\0j \circ \borel\0j$, then
\begin{align*}
\laplace \circ \borel[H]
&=\m({\mathcal I_1 \circ \cdots \circ \mathcal I_{d-1} 
\circ \mathcal I_d }[H] \\
&=\m({\mathcal I_1 \circ \cdots \circ \mathcal I_{d-1} }[H]
    + \m({\mathcal I_1 \circ \cdots \circ \mathcal I_{d-1} }
      [E_{\bm c}\0d] \\
&= \qquad \cdots \qquad \qquad \cdots \qquad \qquad \cdots \qquad \\
&= H + E_{\bm c}\01 + \mathcal I_1[E_{\bm c}\02] + \cdots 
     + \m({\mathcal I_1 \circ \cdots \circ \mathcal I_{d-1}}
      [E_{\bm c}\0d] \,.
\end{align*}
where each $E_{\bm c}\0j$ is an analytic function on $\odom$ which is entire \wrt\ the variable $u_j$. The same is true for 
$\mathcal I_1\circ \cdots \mathcal I_{j-1}[E_{\bm c}\0j]$, because the integral transform $\mathcal I_1\circ \cdots \mathcal I_{j-1}$ does not affect the variable $u_j$.
It follows that $E_{\bm c}\01 + \mathcal I_1[E_{\bm c}\02] + \cdots + \m({\mathcal I_1 \circ \cdots \circ \mathcal I_{d-1}}
 [E_{\bm c}\0d]$ is a demi-entire function on $\odom$. This proves the decomposition $\laplace \circ \borel[H] = H+E_{\bm c}$ in any dimension $d$, and concludes the proof of \Cref{thm:Borel-Laplace}.

\section{Discussions}\label{sec:discussions}

In this section, we discuss some additional properties of the classes of multivariate functions used in the statement of the multivariate transfer theorems.

\paragraph{$\Delta$-analytic functions.}
Like in the univariate case, the multivariate transfer theorems relies fundamentally on the $\Delta$-analyticity of the generating function. But unlike the univariate case, $\Delta$-analyticity is a rather restrictive condition in the multivariate setting. In particular, it implies that the dominant singularity (for any reasonable definition of the term) is unique and independent of the exponent $\bm \theta$ and the direction $\bm \lambda$ of the \diag\ limit taken.
This is in stark contrast with the case of rational functions, where the dominant singularities (a.k.a.\ contributing critical points) generically depend on the direction of the diagonal limit taken.
In fact, a multivariate rational function is never $\Delta$-analytic at any of its poles, unless its denominator has a univariate linear factor which vanishes at this pole.

\note{Recall that a function $A$ is called \emph{meromorphic} at $\bm z^*\in \complex^d$ if there exist a neighborhood $U$ of $\bm z^*$ and analytic functions $F,G:U\to \complex$ such that $G\not \equiv 0$ and $A(\bm z)=\frac{F(\bm z)}{G(\bm z)}$ for all $\bm z \in U$ such that $G(\bm z)\ne 0$. And a function is called meromorphic on an open set $\bm \Omega \subseteq \complex^d$ if it is meromorphic at every point of $\bm \Omega$.
The set of meromorphic functions on $\bm \Omega$ form an integral domain.}

\begin{proposition}[Genuinely multivariate rational functions are never $\Delta$-analytic]
\label{prop:rational D-analytic}
If a rational function is $\Delta$-analytic at $\bm 1$ and has a pole at $\bm 1$, then its denominator is divisible by $z_j-1$ for some $1\le j\le d$.
\end{proposition}

\begin{remark*}
The proposition as well as its proof given below can be easily generalized to meromorphic functions. We shall not enter into the details here and refer to \cite[Section~3.1.1]{Melczer2021} for an introduction to meromorphic functions in several variable.
\end{remark*}

\begin{proof}
Consider a rational function $A=\frac{F}{G}$, where $F,G$ are coprime polynomials. The proposition follows directly from \Cref{lem:rational D-analytic} below about the local geometry of the zero set of a polynomial: it suffices to apply \Cref{lem:rational D-analytic} to the polynomial $H(\bm u)=G(\bm 1-\bm u)$.
\end{proof}

For $\delta,\epsilon>0$, let $\odom[\delta,\epsilon] = (\Omega_{\delta,\epsilon})^d$ with 
$\Omega_{\delta,\epsilon} := \Omega_\delta \cap B_{0,\epsilon} \equiv \Set{u\in \complex}{0<|u|<\epsilon\text{ and }|\arg(u)|<\frac\pi2+\delta}$.

\begin{lemma}[Local version of \Cref{prop:rational D-analytic}]\label{lem:rational D-analytic}
If a polynomial $H\in \complex[\bm u]$ has a zero at $\bm 0$, but no zero on $\odom[\delta,\epsilon]$ for some $\delta,\epsilon>0$, then $H(\bm u)$ is divisible by $u_j$ for some $1\le j\le d$.
\end{lemma}

\begin{proof}
When $d=1$, the lemma is trivial: a univariate polynomial $H(u)$ has a zero at $0$ \Iff\ it is divisible by $u$.

When $d=2$, consider a polynomial $H\in \complex[u,v]$ \emph{not} divisible by $u$ nor $v$, such that $H(0,0)=0$. According to the Newton-Puiseux theorem (see e.g.\ \cite[Corollary~1.5.5, Theorem~1.7.2]{CasasAlvero2000}), there exist an integer $n\ge 1$ and a nonzero analytic function $s$ defined on a neighborhood of $0$, such that for every determination of the $n$-th root, the mapping $\varphi(u)=s(u^{1/n})$ satisfies $H(u,\varphi(u))=0$ for all $u$ in some neighborhood of $0$. Up to swapping the variables $u$ and $v$, we can assume \wlg\ that $\varphi(u) = O(u)$ as $u\to 0$. Since $s$ is analytic at $0$ and is not identically zero, there exist $c\ne 0$ and $\alpha\ge 1$ such that $\varphi(u)\sim c \cdot u^\alpha$ as $u\to 0$. Geometrically, this means that the mapping $\varphi$ multiplies the angles at $u=0$ by $\alpha$. In particular, for any $\delta,\epsilon>0$ with $\epsilon$ small enough, the image $\varphi(\Omega_{\delta,\epsilon})$ conains an angle of $\alpha(\pi+2\delta)$ at $0$. Since $\alpha\ge 1$, we have $\alpha(\pi+2\delta) + (\pi+2\delta) >2\pi$. This implies that $\varphi(\Omega_{\delta,\epsilon}) \cap \Omega_{\delta,\epsilon} \ne \varnothing$. In other words, the graph of the mapping $\varphi$ intersects $(\Omega_{\delta,\epsilon})^2$. It follows that the polynomial $H$ has zeros on $(\Omega_{\delta,\epsilon})^2$ for any $\delta,\epsilon>0$. By contraposition, this proves \Cref{lem:rational D-analytic} when $d=2$.

For $d\ge 3$, we give a proof by contradiction based on the result of the case $d=2$: Let $H\in \complex[\bm u]$ be a polynomial with a zero at $\bm 0$, no zero on $\odom[\delta,\epsilon]$ for some $\delta,\epsilon>0$, and such that $H(\bm u)$ is \emph{not} divisible by $u_j$ for any $1\le j\le d$. 
For $1\le j\le d$ and $m\ge 0$, define $E_m\0j = \set{\bm n \in \natural^d}{(n_1+\cdots+n_d)-n_j=m}$. Let us prove the following statement by induction on $m$:
\begin{equation}
\forall j\in \{1,\ldots,d\},\,
\forall \bm n\in E_m\0j,\quad
[\bm{u^n}] H(\bm u) = 0 \,.
\tag{$\mathcal H_m$}
\end{equation}

For $\bm x\in \real_{>0}^{d-1}$, let $H_{\bm x}(u,v)=H(u,v x_1,\ldots, v x_{d-1})$. One can check that for all $k\in \natural$, we have
\begin{equation}\label{eq:rational D-analytic proof bivariate coefficient}
[v^k] H_{\bm x}(u,v) 
= \sum_{\bm n\in E_k\01} [\bm{u^n}]H(\bm u) 
  \cdot u^{n_1} \cdot x_{1}^{n_2} \cdots x_{d-1}^{n_d} \,.
\end{equation}
Now fix some $m\ge 0$ and assume that $(\mathcal H_k)$ is true for all $0\le k<m$. By \eqref{eq:rational D-analytic proof bivariate coefficient}, we have $[v^k]H_{\bm x}(u,v) =0$ for all $k<m$. Hence $H_{\bm x}\0m(u,v) := v^{-m} H_{\bm x}(u,v)$ is a polynomial in $(u,v)$. Moreover, $H_{\bm x}\0m(0,0)=0$\,: For $m=0$, we have $H_{\bm x}\00(0,0) = H_{\bm x}(0,0)=H(\bm 0)=0$ by assumption. When $m\ge 1$, one can check that
\begin{equation}
\setb{\bm n\in E_m\01}{n_1=0} \ \subseteq \ 
\bigcup_{k<m} \mH({ \bigcup_{j=1}^d E_k\0j }\,.
\end{equation}
Hence the hypotheses $(\mathcal H_k)_{k<m}$ and \eqref{eq:rational D-analytic proof bivariate coefficient} imply that $H_{\bm x}\0m(0,0)\equiv [u^0 v^m]H_{\bm x}(u,v) = 0$ as well.
On the other hand, since $H(\bm u)\ne 0$ for all $\bm u\in \odom[\delta,\epsilon]$, we have $H_{\bm x}\0m(u,v) = v^{-m} H(u,vx_1,\ldots,vx_{d-1}) \ne 0$ for all $(u,v)\in (\Omega_{\delta,\tilde \epsilon})^2$, where $\tilde \epsilon = \min \m({ 1,x_1^{-1},\ldots,x_{d-1}^{-1} } \cdot \epsilon > 0$. 
So, according to the result of the case $d=2$, the polynomial $H_{\bm x}\0m(u,v)$ is either divisible by $u$ or divisible by $v$. 
In particular, we have
\begin{equation}
H_{\bm x}\0m(0,1) \cdot H_{\bm x}\0m(u,0) 
\equiv H(0,\bm x)\cdot [v^m] H_{\bm x}(u,v) = 0
\end{equation}
for all $\bm x\in \real_{>0}^{d-1}$. By analytic continuation, the above identity is valid for all $\bm x\in \complex^{d-1}$. Because $H(\bm u)$ is not divisible by $u_1$, the polynomial $\bm x\mapsto H(0,\bm x)$ is not identically zero. Since the ring of polynomials is an integral domain, this implies that $[v^m]H_{\bm x}(u,v)=0$ as a polynomial in $\bm x$. Comparing this to \eqref{eq:rational D-analytic proof bivariate coefficient}, we see that $[\bm{u^n}]H(\bm u)=0$ for all $\bm n\in E_m\01$. 
The same argument works for $E\0j_m$ for any $1\le j\le d$. Therefore $(\mathcal H_m)$ is true.

By induction, $(\mathcal H_m)$ holds for all $m\ge 0$. But this implies that $H(\bm u)\equiv 0$, which contradicts the assumption that $H(\bm u)$ is not divisible by $u_j$. This completes the proof of \Cref{lem:rational D-analytic} for general $d\ge 3$.
\end{proof}

\paragraph{Demi-analytic functions.}
A univariate function $A$ is demi-analytic at $\rho$ \Iff\ it is analytic in a neighborhood of $\rho$, and it is demi-entire \Iff\ it is an entire function. Thus the names ``demi-analytic'' and ``demi-entire''.

It is easy to check that a \thetahomo\ function is demi-entire \wrt\ a cone \Iff\ it is demi-analytic at $\bm 0$ \wrt\ the same cone.

Given the form of the asymptotic expansion \eqref{eq:m-transfer coeff expansion} and the analyticity of the scaling function $I(\bm \lambda)$, it is not hard to see that the coefficients $\an$ decays exponentially as $n_0\to \infty$ \Iff\ $I$ is identically zero. By \Cref{cor:scaling function properties}, this happens \Iff\ the homogeneous component $H(\bm u)$ is demi-analytic. In this sense, the demi-analyticity condition in \Crefp{1}{thm:m-transfer} is optimal.

\paragraph{Generalized homogeneous functions.}
If $H$ is a \thetahomo\ function, then for all $u_1>0$, we have
\note{If a \thetahomo\ $H$ is analytic on a cone $\bm K\subseteq \complex^d$ with non-empty interior, then by analytic continuation, the scaling relation \eqref{eq:def homogeneous} holds for any $\sigma$ close enough to the positive real axis. In particular, for $u_1$ close enough to the positive real axis, we have}
\begin{equation}
H(\bm u) = u_1^{\theta_0/\theta_1} \cdot H \m({ 1,\frac{u_2}{u_1^{\theta_2/\theta_1}\,}, \cdots, \frac{u_d}{\,u_1^{\theta_d/\theta_1}} } \,.
\end{equation}
Inversely, for any function $h(r_2,\ldots,r_d)$ of $d-1$ variables, $H(\bm u) = u^{\theta_0/\theta_1} \cdot h \mb({ u_1^{-\theta_2/\theta_1} u_2, \cdots, u_1^{-\theta_d/\theta_1} u_d }$ is a \thetahomo\ function satisfying $h(r_2,\ldots,r_d) = H(1,r_2,\ldots,r_d)$.
In particular, when $d=1$, the only $(\theta_0,\theta_1)$-homogeneous analytic functions are the power functions $H(u)=C\, u^{\theta_0/\theta_1}$ ($C\in \complex$).

For each $\bm \alpha \in \real^d$, the monomial $\bm u \mapsto \bm{u^\alpha}$ is \thetahomo\ for all $(\theta_0,\bm \theta)$ such that $\bm{\alpha \cdot \theta} = \theta_0$.
Clearly, a linear combination of finitely many such monomials is also \thetahomo, as long as their multi-exponents $\bm \alpha$ satisfy the above linear relation for the same \thetas.
Such polynomials are obviously \emph{not} all the homogeneous functions, but they provide a good intuition for how a non-homogeneous function could be decomposed into homogeneous components.

\paragraph{Functions of polynomial type.}
In \Cref{thm:m-transfer bis}, the remainder term in the asymptotic expansion of the generating function $A(\bm z)$ was expressed as a big-O of some \thetahomo\ function of polynomial type. The following lemma tells us that every such remainder term can also be bounded by finitely many \emph{monomials} with the same homogeneity.

\begin{lemma}\label{lem:homogeneous monomial bounds}
Let $\bm K$ be a cone in $\complex^d$.
If $H:\bm K\to \complex$ is a \thetahomo\ function of polynomial type locally at $\bm 0$, then there exist a constant $C$ and finitely many vectors $\bm \alpha\0k \in \real^d$, such that $\bm \alpha\0k \bm{\cdot \theta} = \theta_0$ for all $k$, and
\begin{equation}\label{eq:homogeneous monomial bounds}
\forall \bm u \in \bm K,\quad |H(\bm u)| \le C \cdot \sum_k \mb|{ \bm u^{\bm \alpha\0k} } \,.
\end{equation}
\end{lemma}

\begin{proof}
By considering $\mathcal H(\bm s) = H(\bm{s^\theta})$ instead of $H$, we can assume \wlg\ that $\bm \theta=\bm 1$.
By assumption, there exist constants $C,M,\epsilon>0$ such that on $\set{\bm u\in \bm K}{\forall 1\le j\le d,\,|u_j|<\epsilon}$, we have
\begin{equation}
|H(\bm u)| \le C \cdot \m({ |u_1|^{-M} + \cdots + |u_d|^{-M}  } \,.
\end{equation}
Let $i$ be such that $|u_i| = \max_{1\le j\le d} |u_j|$. Then we can rescale the vector $\bm u$ by $\frac{\epsilon}{|u_i|}$ to place it in the above set. Since $H$ is $(\theta_0,\bm 1)$-homogeneous, we have $H(\bm u) = \epsilon^{-\theta_0} |u_i|^{\theta_0} H\m({ \epsilon\frac{u_1}{|u_i|},\,\ldots\, ,\epsilon \frac{u_d}{|u_i|}}$, and therefore
\begin{align*}
\abs{H(\bm u)} &\le
\epsilon^{-\theta_0} |u_i|^{\theta_0}
 C \cdot \mB({ \epsilon^{-M} \mB|{\frac{u_1}{u_i}}^{-M} + \cdots + \epsilon^{-M} \mB|{\frac{u_d}{u_i}}^{-M} }
\\ &=
C \epsilon^{-\theta_0-M} \m({ \mb|{u_i^{\theta_0+M} u_2^{-M}} + \cdots + \mb|{u_i^{\theta_0+M} u_d^{-M}}}
\end{align*}
Summing the \rhs\ over $i$ gives a bound of $H(\bm u)$ of the form \eqref{eq:homogeneous monomial bounds} for all $\bm u \in \bm K$.
\end{proof}

\Cref{lem:homogeneous monomial bounds} was briefly used in the proof of \Cref{thm:m-transfer bis} to obtain the bound \eqref{eq:main proof H(u) bound}. It implies \eqref{eq:main proof H(u) bound} because the variable $\bm u\in \vv$ in \eqref{eq:main proof H(u) bound} satisfies $|u_j|\ge 1$ for all $1\le j\le d$.

In practical examples, it is usually not hard to check that a function is of polynomial type, in particular thanks to the  following lemma.

\begin{lemma}[Closure properties of $\poly(\kdom)$ and $\poly(\ddom)$]\label{lem:poly-type closure}
For any $\delta\in (0,\pi)$, the space $\poly(\kdom)$ forms a $\complex$-algebra \wrt\ pointwise addition and multiplication of functions.\\
Moreover, if $f$ is any function such that $|f(x)|$ is bounded by a polynomial of $|x|$ (e.g.\ $f(x)=x^\beta$, $\beta\in \complex$), then for all $h\in \poly(\kdom)$ such that $f\circ h$ is well-defined and analytic on $\kdom$, we have $f\circ h \in \poly(\kdom)$.

The same is true for $\poly(\ddom)$ with $\delta\in (0,\pi/2)$.
\end{lemma}

\begin{proof}
Fix $\delta\in (0,\pi)$ and let $g,h\in \poly(\kdom)$. By definition there exist $C,M>0$ such that
\begin{align}
\abs{ g(\bm u) } &\le C\cdot \m({ |u_1|^M + |u_1|^{-M} + \cdots + |u_d|^M + |u_d|^{-M} } \\
\tq{and}
\abs{ h(\bm u) } &\le C\cdot \m({ |u_1|^M + |u_1|^{-M} + \cdots + |u_d|^M + |u_d|^{-M} }
\end{align}
for all $\bm u \in \kdom$.
It is clear that any linear combination of $g$ and $h$ satisfies a bound of the same form. So $\poly(\kdom)$ is a $\complex$-vector space. In addition, since $ab\le \frac12(a^2+b^2)$ for all $a,b\in \real$, we have
\begin{align*}
|g(\bm u)h(\bm u)|
& \le C^2 \sum_{i,j=1}^d \mB({
  |u_i|^M |u_j|^M
+ |u_i|^M |u_j|^{-M}
+ |u_i|^{-M} |u_j|^M
+ |u_i|^{-M} |u_j|^{-M}
} \\
& \le C^2 \sum_{i,j=1}^d \mB({
  |u_i|^{2M} + |u_j|^{2M}
+ |u_i|^{-2M}+ |u_j|^{-2M}
} \,.
\end{align*}
So $\poly(\kdom)$ is also closed under multiplication, and therefore forms a $\complex$-algebra.

Now let $f$ be a function such that $|f(x)|$ is bounded by a polynomial of $|x|$. Then there exist $c>0$ and $m\in \integer_{\ge 0}$ such that $|f(x)|\le c\cdot (1+|x|^m)$. If $h\in \poly(\kdom)$ and $f\circ h$ is well-defined and analytic on~$\kdom$, then we have $h^m\in \poly(\kdom)$ by the closure of $\poly(\kdom)$ under multiplication, and therefore
\begin{equation}
\abs{ f\circ h(\bm u) } \le c \cdot (1+|h(\bm u)^m|)
\le c \cdot \mB({ 1 + C\cdot \m({ |u_1|^M + |u_1|^{-M} + \cdots + |u_d|^M + |u_d|^{-M} } }
\end{equation}
for some $C,M>0$. It follows that $f\circ h\in \poly(\kdom)$.
The same proof works for $\poly(\ddom)$.
\end{proof}

Let $\bm \Omega \subseteq \complex^d$. A function $A:\bm \Omega \to \complex$ is \emph{algebraic} if there exists a polynomial $E\in \complex[a,\bm z]$, such that $E(A(\bm z),\bm z)=0$ for all $\bm z\in \bm \Omega$.

\begin{lemma}[Algebraic functions are of polynomial type]\label{lem:poly-type algebraic}
If $A$ is an algebraic function analytic on $\ddom$, then $A\in \poly(\ddom[\delta'])$ for all $\delta' \in (0,\delta)$.
\end{lemma}

\begin{proof}
Let $A$ be an algebraic function which is analytic on $\ddom$. For $\bm z\in \ddom$, let 
\begin{equation}
h(\bm z) 
= \frac1{\mathrm{dist}(\bm z, \complex^d\setminus \ddom)} 
\equiv \m({ \inf_{\bm w \in \complex^d\setminus \ddom} \norm{\bm z-\bm w} }^{-1},
\end{equation} 
where $\norm{\bm z-\bm w}$ is the Euclidean distance between $\bm z$ and $\bm w$, viewed as points in $\real^{2d}$. Let us show that there exist $C,M>0$ such that
\begin{equation}\label{eq:poly-type algebraic proof}
\forall \bm z \in \ddom,\quad
\mb|{ A(\bm z) } \le C\cdot h(\bm z)^M\,.
\end{equation}

\newcommand{\Graph}{\mathtt{Graph}}
\newcommand{\et}{\text{ and }}
\newcommand{\ou}{\text{ or }}
This is a variant of the \emph{\L{}ojasiewicz inequality} in semi-algebraic geometry. See e.g.\ \cite[Chapter~2]{BochnakCosteRoy1998} for an introduction.
Recall that a set $S\subseteq \real^n$ is semi-algebraic \Iff\ it can be defined by a first order logic formula involving only polynomial conditions (i.e.\ equations or inequalities) and quantifiers $\forall/\exists$ over~$\real$. And a function $f:S\to \real^m$ is semi-algebraic \Iff\ $\Graph(f) := \set{(\bm x,f(\bm x)) \in \real^{n+m}}{\bm x\in S}$ is a semi-algebraic set.

It is not hard to write a first order formula that describes the set $\ddom$ (viewed as a subset of $\real^{2d}$). Therefore $\ddom$ is a semi-algebraic set. 
It is well-known that a continuous algebraic function defined on an open semi-algebraic set is always semi-algebraic (see e.g.\ \cite[Theorem~11]{Wakabayashi2008}). Hence the function $A:\ddom \to \complex \equiv \real^2$ is semi-algebraic.
Thanks to general closure properties of the class of semi-algebraic functions (c.f.~\cite[Section~2.2]{BochnakCosteRoy1998}), we deduce from the above facts that the functions $h:\ddom\to \real$ and $|A|:\ddom\to \real$ are also semi-algebraic.

For each $t>0$, let $G_t=\Set{\bm z\in \ddom}{h(\bm z)=t}$ and $g(t) = \sup_{\bm z \in G_t} |A(\bm z)|$, with the convention that $\sup \varnothing=0$. The graph of the function $g:\real_{>0}\to \real$ can be described by a first order formula as follows:
\begin{align*}
\Graph(g) = \Big\{\,(t,y)\in \real^2 \ \Big|\ 
t>0 \et& \m({ \forall \bm z\in G_t,\, |A(\bm z)|\le y } \\
    \et&  \mb({ (G_t=\varnothing \et y=0) 
           \ou (\exists \bm z\in G_t,\, |A(\bm z)|=y) } 
\,\Big\}
\end{align*}
where 
\begin{align*}
\mB({ \forall \bm z\in G_t,\, |A(\bm z)|\le y }
\ \iff\ 
\Big(\, \forall \bm z\in \real^{2d},\, &
        \mb({ \bm z\in \ddom \et (\bm z,t) \in \Graph(h) } \\
        \Rightarrow &
        \mb({ \exists w\in \real,\, (\bm z,w)\in \Graph(|A|) 
                                    \et w\le y } 
\Big) 
\,, \\
\mB({ G_t=\varnothing \et y=0 } \iff\
\Big(\, \forall \bm z \in \real^{2d},\, &
        \mb({ \bm z\in \ddom 
              \Rightarrow (\bm z,t)\not\in \Graph(h) }
\Big) 
\,, \\
\mB({ \exists \bm z\in G_t,\, |A(\bm z)|=y } \iff\
\Big(\, \exists \bm z \in \real^{2d},\, &
        \mb({ \bm z\in \ddom \et (\bm z,t)\in \Graph(h) 
                             \et (\bm z,y)\in \Graph(|A|) }
\Big) \,.
\end{align*}
Since $h$ and $|A|$ are semi-algebraic functions, the conditions $(\bm z,t)\in \Graph(h)$ and $(\bm z,y)\in \Graph(|A|)$ can be further expanded into first order logic formulas involving only polynomial conditions and quantifiers over $\real$. It follows that $g$ is a semi-algebraic
function. A classical result on the growth rate of univariate semi-algebraic functions \cite[Proposition~2.6.1]{BochnakCosteRoy1998} states that there exist constants $C,M,t_0>0$ such that $g(t)\le C\cdot t^M$ for all $t>t_0$. The definition of $g$ implies that it is bounded on $(0,t_0]$. Hence the bound $g(t)\le C\cdot t^M$ extends to all $t>0$. This proves \eqref{eq:poly-type algebraic proof}.

Finally, notice that for all $\delta'\in (0,\delta)$, there exists a constant $c>0$ such that
\begin{equation}
\forall \bm z\in \ddom[\delta'] ,\quad
c^{-1} \cdot \mathrm{dist}(\bm z,\complex^d\setminus \ddom)
\ \le\ \min_{1\le j\le d} |z_j-1| 
\ \le\ c\cdot \mathrm{dist}(\bm z,\complex^d\setminus \ddom)\,.
\end{equation}
Moreover, we have
\begin{equation}
               \m({ \min_{1\le j\le d} |z_j| }^{-M} \!
\le\ |z_1-1|^{-M} + \cdots + |z_d-1|^{-M}
\,\le\, d \cdot \m({ \min_{1\le j\le d} |z_j| }^{-M}.
\end{equation}
It follows that $|A(\bm z)|\le \tilde C \sum_{j=1}^d |z_j-1|^{-M}$ on $\ddom[\delta']$ for some $\tilde C,M>0$. That is, $A\in \poly(\ddom[\delta'])$.
\end{proof}

\paragraph{Background on ACSV and its general strategy.}
The results in paper fall under the topic of \emph{analytic combinatorics in several variables} (ACSV). The remaining paragraphs provide some background on ACSV and where this work stands relative to the others.

In general, analytic combinatorics aims at understanding the enumerative properties of large combinatorial structures through the analytic properties of their generating functions. This is usually done in two steps: First, some (possibly implicit) expression of the generating function must be derived from the definition of the combinatorial structure that it encodes. Then, one studies the generating function as an complex analytic function to derive asymptotic formulas of its coefficients. 

Compared to analytic combinatorics in one variable, to which \mbox{the transfer theorems~\ref{thm:u-transfer}~and~\ref{thm:u-transfer bis} belong}, analytic combinatorics in several variables is a much less mature theory. 
The difference is especially stark when it comes to derving coefficient asymptotics from the generating functions (i.e. the second step outlined above). This process is often known as \emph{singularity analysis}, since the asymptotic expansion of the coefficients of a generating function is mostly determined by the properties of the function near its singularities.
For instance, if a function $A(z)$ has a unique dominant singularity at $\rho \in \complex^*$, then thanks to the analyticity of $A$ everywhere else inside and on the circle of radius $|\rho|$, we can deform the contour of integration in the Cauchy integral formula \eqref{eq:Cauchy integral formula} to a curve $\mathcal C$ that coincides with a circle of radius $r>|\rho|$ everywhere except in a neighborhood $\mathcal N$ of $\rho$. 
By splitting the parts of the contour inside and outside $\mathcal N$, we get $[z^n]A(z) = I\1{loc} + I\1{rem}$ with
\begin{equation}
I\1{loc} = \frac{1}{2\pi i}\int_{\mathcal C \,\cap\, \mathcal N} \frac{A(z)}{z^{n+1}}\dd z
\qtq{and}
I\1{rem} = \frac{1}{2\pi i}\int_{\mathcal C \,\setminus\, \mathcal N} \frac{A(z)}{z^{n+1}}\dd z \,.
\end{equation}
Since $|z|=r$ on $\mathcal{C\setminus N}$, we have $I\1{rem} = O(r^{-n})$, which is exponentially small compared to $|\rho|^{-n}$. On the other hand, $\limsup_{n\to \infty} ([z^n]A(z) )^{1/n} = |\rho|^{-1}$ by the root test of radius of convergence. Hence the asymptotics of $[z^n]A(z)$ is dominated by the term $I\1{loc}$.
Since $I\1{loc}$ only depends on $A(z)$ in an arbitrarily small neighborhood of $\rho$, its asymptotics can be further studied via asymptotic expansions of the function $A(z)$ near its singularity $z=\rho$. (This is basically the beginning of a proof of Theorem~\ref{thm:u-transfer bis}.)

At first glance, it is not clear how the above approach of singularity analysis could be generalized to multivariate functions. Indeed, the singularities of a multivariate complex function $A(\bm z)$ always form a continuous set with no isolated points (c.f.\ Hartog's extension theorem). 
The crucial remark here is that in general, not all singularities of $A(\bm z)$ contribute to the dominant asymptotics of its coeffcients. In nice cases, one can even expect to find a finite number of \emph{contributing singularities}, so that the asymptotics of $[\bm{z^n}]A(\bm z)$ is dominanted by the values of $A(\bm z)$ in arbitrarily small neighborhoods of these points. In practice, one would like deform the torus $\bm T_r$ in the Cauchy integral formula \eqref{eq:Cauchy integral formula} to a cycle (i.e.\ $d$-chain without boundary) $\cc$ homologous to $\bm T_r$ in the domain of analyticity of the integrant $\frac{A(\bm z)}{\bm z^{\bm n+\bm 1}}$, such that the denominator $\abs{\bm z^{\bm n+\bm 1}}$ attains its mimimum only at a finite number of points $\bm \rho\01, \ldots, \bm \rho\0m$ on $\bm C$. By Stokes' theorem (see e.g.~\cite[Appendix~A.2]{PemantleWilson2013}), such a deformation does not change the value of the integral. One can then take arbitrarily small neighborhoods $\bm{\mathcal N}\01,\ldots \bm{\mathcal N}\0m$ of the points $\bm \rho\01, \ldots, \bm \rho\0m$, and split the integral inside and outside these neighborhoods as in the univariate setting. This gives the decomposition
$[\bm{z^n}]A(\bm z) = I\1{loc}\01 + \cdots + I\1{loc}\0m + I\1{rem}$, where
\begin{equation}
I\1{loc}\0s = \m({ \frac1{2\pi i} }^d \!\int_{\cc\0s} \frac{A(\bm z)}{\bm z^{\bm n + \bm 1}} \dd \bm z
\qtq{and}
I\1{rem} = \m({ \frac1{2\pi i} }^d \!\int_{\bm C\setminus \m({ \cc\01 \cup \cdots \cup \cc\0m }} \frac{A(\bm z)}{\bm z^{\bm n + \bm 1}} \dd \bm z \,,
\end{equation}
with $\cc\0s := \cc \cap \bm{\mathcal N}\0s$.
For a well chosen cycle $\cc$, we expect the \emph{non-local} term $I\1{rem}$ to be exponentially small compared to the \emph{local} terms $I\1{loc}\0s$, so that the latter dominates the the asymptotics of $[\bm{z^n}]A(\bm z)$. After this, one still needs to expand the localized integrals $I\1{loc}\0s$ to obtain a simple asymptotic expansion of $[\bm{z^n}]A(\bm z)$. How this can be done depends on the form of the function $A(\bm z)$. But usually this is a simpler problem, since there are a lot more tools at our disposal for the local analysis of a function.

The above general strategy to the multivariate singularity analysis has been outlined by Pemantle and Wilson in \cite{PemantleWilson2013}. Over the past twenty years, they and their collaborators developped an impressive theory that treats the singularity analysis of multivariate rational-type functions in an algorithmic way. The results of their project are collected at \url{http://acsvproject.com/}. More precisely, they consider general rational functions in several variables (and also some other functions whose singularities form algebraic varieties) and compute their coefficient asymptotics in the diagonal limit (i.e.~$\bm \theta$-diagonal limit with $\bm \theta=(1,\ldots,1)$). In this broad setting, they carry out the strategy described in the previous paragraph in a systematic way, identifying the contributing singularities $\bm \rho\0s$ and simplifying the localized integrals $I\1{loc}\0s$, using powerful tools from algebraic topology and Morse theory.
The reason that they need such advanced tools is that the singular set of a general rational function can have a quite complicated geometry. In particular, the location of the dominant singularities $\bm \rho\01,\ldots,\bm \rho\0m$ (or \emph{contributing critical points}, as they are called in \cite{PemantleWilson2013}) depends on the direction $\bm \lambda$ of the diagonal limit. And in general, they cannot be reached by the easy-to-visualize cycles of product form $\cc= \cj[1]\times \cdots \times \cj[d]$.

This paper explores the case of $\Delta$-analytic generating functions, following the same general strategy of singularity analysis. As shown by \Cref{prop:rational D-analytic}, this case is essentially disjoint from that of rational functions. \mbox{$\Delta$-analytic} functions have a simpler singularity structure, in the sense that they always have a unique, fixed dominant singularity reachable by a cycle of product form. This allows us to study their coefficients in the $\bm \theta$-diagonal limits for  general $\bm \theta \in \real_{>0}^d$, while using mostly elementary tools from univariate complex analysis. We will discuss more the relation between our case and the case of rational functions in \Cref{sec:discussions}. For more background on ACSV, we refer to the historical accounts in \cite[Chapter~1]{PemantleWilson2013} and \cite[Section~1.2.2]{Melczer2021}.

\bibliographystyle{abbrv}
\bibliography{/Users/linxiao/tex-lib/bibdata-2022-01-05}
\Addresses

\end{document}